\documentclass[letter,11pt]{amsart}
\usepackage{amsmath,amsthm,amssymb,amsfonts}
\usepackage[pdftex]{graphicx}
\usepackage{tikz}
\usepackage[margin=1in,top=1in]{geometry}
\usepackage[latin1]{inputenc}
\usepackage{enumitem}
\usepackage{hyperref}

\usepackage{etoolbox}
\patchcmd{\footnotemark}{\stepcounter{footnote}}{\refstepcounter{footnote}}{}{}

\newcommand{\R}{\mathbb{R}}
\newcommand{\vep}{\varepsilon}

\newcommand{\tr}[1]{\operatorname{Tr}(#1)}
\newcommand{\chara}[1]{\chi_{#1}}
\renewcommand{\div}{\operatorname{div}}
\newcommand{\id}{I}

\newcommand{\ellmat}{\mathcal{A}_{\lambda, \Lambda}}
\newcommand{\ellvec}{\mathcal{O}_{\lambda,\Lambda,\mu}}
\newcommand{\env}{\mathcal{O}}
\newcommand{\mcal}{\mathcal{M}}
\newcommand{\ncal}{\mathcal{N}}
\newcommand{\tcal}{\mathcal{T}}
\newcommand{\scal}{\mathcal{S}}

\newcommand{\usc}{u.s.c.\,}
\newcommand{\lsc}{l.s.c.\,}
\newcommand{\uppers}{USC}
\newcommand{\lowers}{LSC}
\newcommand{\subsol}{\mathcal{U}}

\newtheorem{thm}{Theorem}[section]
\newtheorem{prop}[thm]{Proposition}
\newtheorem{cor}[thm]{Corollary}
\newtheorem{lem}[thm]{Lemma}
\theoremstyle{definition}
\newtheorem{defn}[thm]{Definition}
\newtheorem{rem}[thm]{Remark}

\allowdisplaybreaks

\numberwithin{equation}{section}

\author[I. U. Erneta]{I\~{n}igo U. Erneta}
    \address{I\~{n}igo U. Erneta. Department of Mathematics\\
    Rutgers University\\
110 Frelinghuysen Rd., Piscataway, NJ 08854, USA}
    \email{iu40@math.rutgers.edu}

\title{Fully nonlinear oblique transmission problems: Well-posedness}

\begin{document}

\begin{abstract}
We introduce elliptic transmission problems involving fully nonlinear equations in both the interior bulk equations and the interface transmission condition.
In contrast to all previous works, our transmission condition depends on both the normal and tangential derivatives of the solution from each side of the interface.
Our main result establishes the existence and uniqueness of viscosity solutions for constant-coefficient oblique transmission problems with flat interfaces.
\end{abstract}

\maketitle

\tableofcontents

\section{Introduction}

Transmission problems provide a general mathematical framework to describe phenomena across heterogeneous domains.
Specific examples from the applied sciences include 
the modeling of composite materials in Physics, the description of diffusion across cell membranes in Biology, or regime-switching models in Economics.
In each of these problems, the configuration space is made up of several components with different properties.
While a variational ``divergence-form'' theory for such problems is by now classical, an equally satisfactory non-variational approach has yet to be established.
The present paper develops a methodology  aimed at filling this longstanding gap.

\subsection{Framework and motivation}
Given smooth bounded open connected sets $\Omega^{+} \subsetneq \Omega \subseteq \R^n$, we define $\Omega^{-} := \Omega \setminus \overline{\Omega^{+}}$ as the complement of $\overline{\Omega^{+}}$ inside $\Omega$.
The two \emph{phases} $\Omega^{\pm}$ are separated by the \emph{interface} $\Gamma = \partial \Omega^{+} \cap \Omega$, which is a smooth hypersurface, and we have the partition $\Omega = \Omega^{+} \cup \Gamma \cup \Omega^{-}$.
For ``global'' functions $u \colon \Omega \to \R$, we use the superscript notation
$u^{\pm} := u \chara{\Omega^{\pm} \cup \Gamma}$
to denote the restrictions to each phase, up to the interface.

\medskip

We are interested in the transmission problem
\begin{equation}
\label{gen:dirichlet}
\begin{cases}
F^{+}(D^2u, x) = f^{+}(x) & \text{ in } \Omega^{+},\\
F^{-}(D^2u, x) = f^{-}(x) & \text{ in } \Omega^{-},\\
G(\nabla u^{+}, \nabla u^{-}, x) =g(x) & \text{ on } \Gamma,\\
u^{+} = u^{-} & \text{ on } \Gamma,\\
u = \phi & \text{ on } \partial \Omega,
\end{cases}
\end{equation}
where $F^{\pm}$ and $G$ are (nonlinear) operators, $f^{\pm}$ and $g$ are functions in $\Omega^{\pm}$ and $\Gamma$, respectively, and $\phi$ is a boundary datum on $\partial \Omega$.
The second order equations in $\Omega^{\pm}$ are called \emph{bulk equations}, while the first order equation on $\Gamma$ is known in the literature as a \emph{transmission condition}.
The second equation on $\Gamma$ simply expresses the continuity of $u$ across the interface and will be assumed implicitly throughout the text.

\medskip

Problem~\eqref{gen:dirichlet} arises naturally when studying stochastic processes across multiple regimes.
As a guiding example, consider a diffusion process $\{X_t\}_{t \geq 0}$ in the two-phase domain $\Omega$.
Suppose that, upon reaching the interface $\Gamma$, the process can undergo either of the following behaviors:
\begin{itemize}
\item reflect back into the phase from which it came;
\item transmit across the interface into the adjacent phase;
\item propagate along the interface in a prescribed tangential direction.
\end{itemize}
Let $\nu$ denote the unit normal vector to the interface $\Gamma$, pointing into $\Omega^{+}$.
Assume that a payoff $\phi$ is received when the process $X_t$ exits the domain through $\partial\Omega$, at which point the process is stopped.
Then, at least heuristically, the expected payoff $u(x)$ associated with a process starting at $x \in \overline{\Omega}$ should solve~\eqref{gen:dirichlet}, together with a linear transmission condition
\begin{equation}
\label{trans1}
\gamma_{\nu}^{+} \partial_{\nu}u^{+} - \gamma_{\nu}^{-} \partial_{\nu} u^{-} 
+ \gamma_{\Gamma} \cdot \nabla_{\Gamma} u 
= 0 
\quad \text{ on } \Gamma,
\end{equation}
where the coefficients can be interpreted as follows:
\begin{itemize}
\item $\gamma_{\nu}^{+} >0$ is the probability of moving into $\Omega^{+}$ in the normal direction $\nu$;
\item $\gamma_{\nu}^{-} > 0$ is the probability of moving into $\Omega^{-}$ in the normal direction $-\nu$;
\item $\gamma_{\Gamma}$ is a tangent vector to the hypersurface $\Gamma$ and $|\gamma_{\Gamma}| = 1 - \gamma_{\nu}^{+} - \gamma_{\nu}^{-}$ is the probability of sliding along $\Gamma$ in the tangential direction $\gamma_{\Gamma}$.
\end{itemize}
Note here that since $u$ is continuous across the interface, assuming $u^{+}$ and $u^{-}$ to be differentiable up to $\Gamma$, the tangential gradient $\nabla_{\Gamma} u := \nabla [ u|_{\Gamma}]$ is well defined and independent of the restriction.

\medskip

It is more convenient to instead write the linear transmission~\eqref{trans1} in the form
\begin{equation}
\label{trans2}
\gamma^{+} \cdot \nabla u^{+} - \gamma^{-} \cdot \nabla u^{-} = 0 \quad \text{ on } \Gamma,
\end{equation}
where $\gamma^{\pm} = \gamma^{\pm}(x)$ are vectors in $\R^n$ with
\begin{equation}
\label{prob:ell}
\gamma^{\pm} \cdot \nu > 0 \quad \text{ on } \Gamma.
\end{equation}
Thus, the transmission condition~\eqref{trans2} can be interpreted as a coupling of oblique derivatives from each side of $\Gamma$.
Assumption~\eqref{prob:ell} is a natural ellipticity condition which prevents the process from degenerating to a lower dimensional motion on the interface. 

\medskip

Our model~\eqref{gen:dirichlet} encompasses most problems of interest, such as:
\begin{itemize}
\item \textbf{Expected payoff}.
Assuming a linear anisotropic diffusion on each phase, the expected payoff satisfies the linear transmission problem
\[
\begin{cases}
\tr{A^{\pm}(x) D^2 u} = f^{\pm}(x) & \text{ in } \Omega^{\pm},\\
\gamma^{+}(x) \cdot \nabla u^{+}
- \gamma^{-}(x) \cdot \nabla u^{-}
 = g(x) & \text{ on } \Gamma,
\end{cases}
\]
where $f^{\pm}$ is the running cost of remaining in $\Omega^{\pm}$, while $g$ is the ``toll'' for crossing $\Gamma$.
\item \textbf{Stochastic control}.
For a one-parameter family of stochastic processes $\{X_t^{\alpha}\}$, the maximum expected gain when choosing a control for $\alpha$ will satisfy the linear programming or Bellman-type equations
\[
\begin{cases}
\displaystyle
\sup_{\alpha} \, 
 \Big[
\tr{A_{\alpha}^{\pm}(x) D^2 u} - f_{\alpha}^{\pm}(x)
 \Big] = 0
& \text{ in } \Omega^{\pm},\\
\displaystyle
\sup_{\alpha}  \, \Big[
\gamma^{+}_{\alpha}(x) \cdot \nabla u^{+}
- \gamma^{-}_{\alpha}(x) \cdot \nabla u^{-}
- g_{\alpha}(x)
\Big] = 0
& \text{ on } \Gamma.
\end{cases}
\]
\item \textbf{Differential games}.
For two-player games with strategies described by two-parameter processes $\{X^{\alpha\beta}_{t}\}$, the optimal payoff for a player satisfies the Isaacs-type equations
\[
\begin{cases}
\displaystyle
\inf_{\beta} \sup_{\alpha}
\, \left[
\tr{A_{\alpha\beta}^{\pm}(x) D^2 u} - f_{\alpha\beta}(x)
 \right] = 0
& \text{ in } \Omega^{\pm},\\
\displaystyle
\inf_{\beta}\sup_{\alpha}  \, \left[
\gamma^{+}_{\alpha\beta}(x) \cdot \nabla u^{+}
- \gamma^{-}_{\alpha\beta}(x) \cdot \nabla u^{-}
- g_{\alpha\beta}(x)
\right] = 0
& \text{ on } \Gamma.
\end{cases}
\]
\end{itemize}

\medskip

All linear operators featured in the nonlinear problems above should satisfy appropriate ellipticity assumptions.
These conditions involve three natural ellipticity constants:
$0 < \lambda \leq \Lambda$ and $\mu \geq 0$.
Namely, the symmetric matrices $A^{\pm}(x)$ are assumed to be uniformly elliptic with
\[
\lambda \id \leq A^{\pm}(x) \leq \Lambda \id \qquad \text{ for } x \in \Omega^{\pm},
\]
while the transmission coefficient vectors $\gamma^{\pm}(x)$ are \emph{uniformly oblique}, i.e., they have uniformly elliptic normal components
\[
\lambda \leq \gamma^{\pm}_{\nu}(x) := \gamma^{\pm}(x) \cdot \nu(x) \leq \Lambda \qquad \text{ for } x \in \Gamma,
\]
and bounded tangential ones
\[
|\gamma^{\pm}(x) - \gamma^{\pm}_{\nu}(x) \nu(x)| \leq \mu \qquad \text{ for } x \in \Gamma.
\]
Our nonlinear operators $G$ satisfy analogous ellipticity conditions, as explained in Section~\ref{sec:visc}.

\medskip

From the point of view of applications, it is also interesting to allow obliqueness to degenerate on one side of the interface.
Namely, we will consider nonlinear transmission operators of the form
\[
G(\nabla u^{+}, \theta \nabla u^{-}, x)
\]
where $0 \leq \theta \leq 1$ is a constant \emph{interpolation parameter}.
Note that  the transmission condition on $u^{-}$ seems to disappear when $\theta$ vanishes.
Somewhat surprisingly, our results and estimates below are all independent of this parameter.
We give the intuition behind this phenomenon in Section~\ref{sec:visc} (see~Proposition~\ref{prop:deg} and Remark~\ref{rem:intuition}).

\medskip

\subsection{State of the art}

Our present work is concerned with the well-posedness of the above models, a fundamental question dating back to the seminal work of Picone~\cite{P}.

\medskip

The mathematical analysis of transmission problems has a rich history, particularly within the framework of linear divergence-form operators.
The classical variational theory was initiated by J.-L.~Lions~\cite{JLL}, where solutions are understood in the weak sense via energy formulations. 
In variational configurations, the transmission condition across the interface involves matching the conormal derivatives
\[
A^{+}(x)\nabla u^{+} \cdot \nu - A^{-}(x)\nabla u^{-} \cdot \nu = g(x) \quad \text{on } \Gamma,
\]
which arises naturally from integration by parts.
We note that conormal derivatives are particular cases of oblique derivatives.
For further details on variational problems, we refer the reader to the monograph of Borsuk~\cite{B} and the references therein.

\medskip

In the non-variational setting, the study of boundary value problems with oblique derivatives has its modern roots in the viscosity theory for the Neumann problem. 
The well-posedness for non-divergence form equations with oblique boundary conditions was established by Lieberman~\cite{LPerron} by adapting Perron's method; see also~\cite{L} and compare with the work of Li-Zhang~\cite{LZ} on fully nonlinear equations.

\medskip

The application of viscosity solutions to transmission problems, however, is a more recent development.
To the best of our knowledge, there are only three previous works dealing with this subject.
Motivated by free boundary problems, De Silva-Ferrari-Salsa~\cite{DFSfb, DFS} pioneered this direction by establishing the viscosity framework.
Subsequently, Soria Carro-Stinga~\cite{SS} developed an existence and uniqueness theory of viscosity solutions.
All these works share a fundamental limitation: the transmission operator $G$ is assumed to be linear and independent of the tangential derivatives of the solution, i.e., it is a coupling of  Neumann conditions.
For convenience, we refer to such problems as \emph{Neumann transmission problems} to distinguish them from our oblique ones.

\medskip

The results of our paper constitute the first treatment of transmission problems with genuinely oblique interface conditions, where the transmission operator depends nonlinearly on both the normal and tangential components of the gradient on each side of the interface.
This setting introduces significant geometric and analytical challenges, as the interface condition can no longer be simplified or decoupled from the bulk equations through standard barrier arguments.

\subsection{Main result}

In this first paper, we focus on the case when $\Gamma$ is flat, while curved interfaces will be treated in a future work.
It is convenient to introduce a special notation for the flat case.
Below, we write $B_R = \{x \in \R^{n} \colon |x| < R\}$ for the ball of radius $R > 0$ centered at the origin.
The half-balls are denoted by
\[
B_{R}^{+} := B_R \cap \{x_n > 0\} \qquad \text{ and } \qquad B_{R}^{-} := B_R \cap \{x_n < 0\},
\]
while the associated flat interface is written
\[
T_{R} := B_R \cap \{x_n = 0\}.
\]

\medskip

Our main result establishes the well-posedness of the Dirichlet problem for constant-coefficient transmission problems.
For the exact definition of ellipticity and viscosity solution, see Section~\ref{sec:visc}:

\begin{thm}[Existence and uniqueness]
\label{thm:wellposed}
Let $\{F^{\pm}, G\}$ be constant-coefficient uniformly elliptic operators.
Let $f^{\pm} \in C(\overline{B}_1^{\pm})$, $g \in C(\overline{T}_1)$, and $\phi \in C(\partial B_1)$.
Let $0 \leq \theta \leq 1$ be a constant.

Then there exists a unique viscosity solution $u \in C(\overline{B}_1)$ to the transmission problem
\[
\begin{cases}
F^{\pm}(D^2u) = f^{\pm}(x) & \text{ in } B_1^{\pm},\\
G(\nabla u^{+}, \theta \nabla u^{-}) =g(x) & \text{ on } T_1,\\
u = \phi & \text{ on } \partial B_1.
\end{cases}
\]
\end{thm}

\medskip

We defer regularity questions to our companion paper~\cite{E}, where we obtain optimal piecewise $C^{1,\alpha}$ regularity of viscosity solutions.
The above existence and uniqueness result is crucial for the regularity theory of variable-coefficient problems.
There, we use Theorem~\ref{thm:wellposed} to construct barriers satisfying suitable constant-coefficient inhomogeneous transmission problems, which play a key role in the perturbation theory from~\cite{E}.

\medskip

We obtain the unique viscosity solution by Perron's method, adapting the approach of Soria~Carro-Stinga~\cite[Section 4]{SS}.
This technique requires three basic ingredients: an Alexandrov-Bakelman-Pucci (ABP) maximum principle, a comparison principle for semicontinuous functions, and the construction of appropriate barriers.
Next, we comment briefly on each of these elements.

\medskip

The classical ABP estimate for non-divergence form equations follows by a well-known geometric argument involving linear polynomials that touch the solution from one side.
For constant-coefficient Neumann transmission problems, the proof of the analogous result can be reduced to the classical one by subtracting a simple barrier.
By contrast, the proof of our ABP maximum principle for genuinely oblique transmission problems (Theorem~\ref{thm:abp}) is more involved and requires a new argument based on piecewise linear functions (see Section~\ref{sec:abp}).
Moreover, our ABP estimate also covers variable-coefficient transmission conditions and is therefore essential to the regularity theory for such problems in~\cite{E}.

\medskip

A comparison principle for Neumann transmission problems had already appeared, first for continuous functions in~\cite{DFSfb,DFS}, and then for semicontinuous ones in~\cite{SS}.
Our proof of the corresponding result for oblique problems (Theorem~\ref{thm:comp}) follows the same strategy,
combining a delicate regularization procedure with the ABP estimate, and then invoking the punctual boundary regularity for fully nonlinear PDE of Ma-Wang~\cite{MW}.
We point out that some additional care is required at one technical point of the original argument (see Remark~\ref{rem:mistake}).
Once the comparison principle is established, uniqueness follows.

\medskip

We need barriers to show that the boundary condition is attained.
For this, we seek inspiration in classical barriers for the oblique derivative problem, which can be understood as a transmission problem with $\theta = 0$ (see Proposition~\ref{prop:deg}).
The most delicate point corresponds to those boundary points in the closure of the interface.
Here, we construct piecewise smooth barriers that extend Lieberman's~\cite{L} barrier based on earlier work of Miller~\cite{Miller}.
The easier case away from the interface can be treated by modifying a simple barrier of Li-Zhang~\cite{LZ}.

\medskip

To conclude, we remark that we do not yet have an existence and uniqueness result analogous to Theorem~\ref{thm:wellposed} for variable-coefficient transmission problems.
The only missing ingredient is a comparison principle, as our regularization procedure does not seem to extend to variable coefficients.
Such a result might perhaps be obtained by adapting the method of doubling variables described in the notes of Crandall-Ishii-Lions~\cite{CIL}; however, we were unable to implement this approach.
We refer the reader to~\cite{CIL} and the references therein for further background.

\subsection{Notation}

In general, given an open set $\Omega \subset \R^n$, we denote its bulk phases by
\[
\Omega^{+} = \Omega \cap \{x_n > 0\}
\qquad \text{ and }
\qquad \Omega^{-} = \Omega \cap \{x_n < 0\},
\]
and its flat interface by
\[
T =
\Omega \cap \{x_n = 0\} = \partial \Omega^{\pm} \cap \Omega.
\]
For $x_0 \in \R^{n}$ and $R> 0$, we denote the open balls in $\R^{n}$ by
\[
B_{R}(x_0) = \{x \in \R^n \colon |x - x_0| < R\} \qquad \text{ and } \qquad B_R = B_{R}(0).
\]
When $x_0 \in T$, we also define the open balls on the interface as
\[
T_R(x_0) = B_R(x_0) \cap \{x_n = 0\}
\qquad
\text{ and }
\qquad T_R = T_R(0).
\]

\medskip

For a continuous function $u$ in $\Omega$, we define its restrictions to each phase, up to the interface, by
\[
u^{+} := u \chara{\overline{\Omega}^{+} \cap \Omega} = u \chara{\Omega^{+} \cup T}
\qquad \text{ and } \qquad 
u^{-} := u \chara{\overline{\Omega}^{-} \cap \Omega} = u \chara{\Omega^{-} \cup T },
\]
where $\chara{E}$ is the characteristic function of the set $E \subset \R^{n}$.
The positive and negative parts of $u$ are
\[
u_{+} := \max\{u, 0\} \qquad \text{and } \qquad u_{-} := \max\{-u, 0\}.
\]
Note our use of superscripts for restrictions, while subscripts denote the positive and negative parts.

\medskip

We write vectors in terms of tangential and normal components as $x = (x', x_n) \in \R^{n-1} \times \R$.

\medskip

Derivatives are written as subscripts, e.g., $u_{x_i}$, $u_{x_i x_j}$, and so on.
We also use the gradient vector $\nabla u = (u_{x_1}, \ldots, u_{x_n})$
and the tangential gradient $\nabla' u = (u_{x_1}, \ldots, u_{x_{n-1}}, 0)$.

\medskip

The space of symmetric $n \times n$ matrices is denoted $\scal_n \subset \R^{n\times n}$.
The set of symmetric matrices $A \in \scal_n$ satisfying $\lambda \id \leq A \leq \Lambda \id$ is written $\ellmat$. 
For matrices $N \in \R^{n\times n}$, we use the operator norm
\[
\|N\| = \sup_{x \in B_1 \subset \R^{n}} |N x|.
\]

The set of upper semicontinuous (u.s.c.) functions $u \colon \Omega \to \R$ is denoted by $\uppers(\Omega)$,
and the set of lower semicontinuous (l.s.c.) ones by $\lowers(\Omega)$.

\medskip

A constant is \emph{universal} if it depends only on $n$, $\lambda$, $\Lambda$, and $\mu$.
We note in particular that a universal constant does not depend on the interpolation parameter $0 \leq \theta \leq 1$.

\subsection{Outline}
Section~\ref{sec:visc} introduces the framework of viscosity solutions.
In Section~\ref{sec:abp} we prove an ABP maximum principle for oblique transmission problems.
Section~\ref{sec:unique} is devoted to establishing a comparison principle, uniqueness being a corollary.
In Section~\ref{sec:perron} we construct the unique viscosity solution via Perron's method.

\medskip


\section{Viscosity solutions to transmission problems}
\label{sec:visc}

Let $\Omega \subset \R^n$ be an open set.
For a constant $0 \leq \theta \leq 1$, we consider the transmission problem
\begin{equation}
\label{eq:nl}
\begin{cases}
F^{\pm} (D^2 u, x) = f^{\pm}(x) & \text{ in } 
\Omega^{\pm},\\
G(\nabla u^{+}, \theta \nabla u^{-}, x) = g(x)
 & \text{ on } T,
\end{cases}
\end{equation}
where the operators $F^{\pm}$ and $G$ are \emph{uniformly elliptic}, in the sense that there are constants $0 < \lambda \leq \Lambda$ and $\mu \geq 0$ such that:
\begin{itemize}
\item For any symmetric matrices $M$ and $N$, with $N$ nonnegative definite, we have that
\[ 
\lambda \|N\| \leq F^{\pm}(M + N,x) - F^{\pm}(M,x) \leq \Lambda \|N\|;
\]
\item
For any vectors $\xi^{\pm}$ and $\eta^{\pm} = (\eta'^{\pm}, \eta_n^{\pm})$ in $\R^{n} = \R^{n-1} \times \R$, with $\eta_n^{\pm} \geq 0$, we have that
\[
\lambda \eta^{+}_n - \Lambda \eta^{-}_n 
- \mu (|\eta'^{+}| + |\eta'^{-}|)
\leq G(\xi^{+} + \eta^{+}, \xi^{-} + \eta^{-}, x) - G(\xi^{+}, \xi^{-}, x) \leq \Lambda \eta^{+}_n - \lambda \eta^{-}_n 
+ \mu( |\eta'^{+}| + |\eta'^{-}|).
\]
\end{itemize}

\medskip

\begin{rem}
In particular, the function $G(\xi^{+}, \xi^{-}, x)$ is increasing in $\xi_n^{+}$ and decreasing in $\xi_n^{-}$.
Essentially, our ellipticity condition quantifies this monotonicity in the normal components as well as the continuity in the tangential ones $\xi'^{\pm}$.
\end{rem}

\medskip

Throughout the paper, we always assume that $F^{\pm}(M, \cdot)$ and $f^{\pm}$ are continuous up to the interface, on $\Omega^{\pm} \cup T = \overline{\Omega}^{\pm} \cap \Omega$.
We also assume that $G(\xi^{+}, \xi^{-}, \cdot)$ and $g$ are continuous on $T$.

\medskip

In Section~\ref{sec:perron}, we construct solutions to~\eqref{eq:nl} via Perron's method.
This approach naturally leads to solutions in the viscosity sense, which are only continuous a priori.
When applying this strategy, it is also necessary to consider viscosity sub- and supersolutions to~\eqref{eq:nl} defined for merely semicontinuous functions.

\medskip

A continuous function $\varphi$ \emph{touches} $u \in \uppers(\Omega)$\emph{ from above} (resp. $u \in \lowers(\Omega)$ \emph{from below}) at a point $x_0 \in \Omega$ if $\varphi(x_0) = u(x_0)$ and 
\[
\varphi(x) \geq u(x) \qquad \text{(resp. } \varphi(x) \leq u(x))
\]
for all $x$ in a neighborhood of $x_0$.
It touches \emph{strictly} if the inequalities are strict for $x \neq x_0$.

\medskip

\begin{defn}[Admissible function]
We say that $\varphi \in C(\Omega)$ is \emph{admissible} if its restrictions $\varphi^{\pm}$ are $C^2$ in the interior and $C^{1,1}$ up to the interface, i.e., if
\[
\varphi^{\pm} \in C^2(\Omega^{\pm}) \cap C^{1,1}(\Omega^{\pm} \cup T).
\]
Note in particular that the tangential gradients match on the interface, i.e., $\nabla' \varphi^{+} = \nabla' \varphi^{-}$ on $T$.
\end{defn}

This class of piecewise $C^{1,1}$ test functions yields the following notion of viscosity solution:

\begin{defn}[Viscosity solutions]
We say that $u \in \uppers(\Omega)$ (resp. $u \in \lowers(\Omega)$) is a \emph{viscosity subsolution} (resp. \emph{supersolution}) to~\eqref{eq:nl} if whenever an admissible function $\varphi$ touches $u$ from above (resp. below) at $x_0 \in \Omega$ and:
\begin{itemize}
\item $x_0 \in \Omega^{\pm}$, then
\[
F^{\pm}(D^2 \varphi(x_0), x_0) \geq f^{\pm}(x_0) \qquad (\text{resp. } F^{\pm}(D^2 \varphi(x_0), x_0) \leq f^{\pm}(x_0));
\]
\item $x_0 \in T$, then
\[
G(\nabla \varphi^{+}(x_0), \theta \nabla \varphi^{-}(x_0), x_0) \geq g(x_0) 
\qquad 
\text{(resp. } 
G(\nabla \varphi^{+}(x_0), \theta \nabla \varphi^{-}(x_0), x_0) \leq g(x_0) ).
\]
\end{itemize}
We say that $u \in C(\Omega)$ is a \emph{viscosity solution} to~\eqref{eq:nl} if it is both a viscosity subsolution and supersolution.
\end{defn}

\begin{rem}
In the definition above, we may replace admissible functions by \emph{piecewise quadratic functions}, i.e., continuous functions $P$ such that each restriction $P^{\pm}$ is a second order polynomial.
Indeed, suppose that an admissible function $\varphi$ touches $u$ from (say) above at $x_0 \in T$.
For $x \in \Omega^{\pm}$ close to $x_0$, we have that
\[
\begin{split}
\varphi(x) &= \varphi(x_0) + \nabla \varphi^{\pm}(x_0) \cdot (x-x_{0}) + \int_{0}^{1} \big[\nabla \varphi^{\pm}((1-t)x_0 + t x) - \nabla \varphi^{\pm}(x_0)\big] \cdot (x- x_0) dt\\
& \qquad \leq \varphi(x_0) + \nabla \varphi^{\pm}(x_0) \cdot (x-x_{0}) + C |x-x_0|^{2} =: P(x),
\end{split}
\]
where $C > 0$ depends on the Lipschitz seminorms of $\nabla \varphi^{\pm}$.
The piecewise quadratic function $P$ satisfies $\nabla P^{\pm}(x_0) = \nabla \varphi^{\pm}(x_0)$ and touches $u$ from above at $x_0$, whence the equivalence follows.
\end{rem}

\begin{rem}
In the subsolution (resp. supersolution) condition, instead of touching $u$ from above (resp. below), it is equivalent to impose that the inequality holds whenever $u - \varphi$ has a local maximum (resp. minimum) at $x_0 \in \Omega$.
\end{rem}

\medskip

The advantage of considering test functions with bounded second derivatives lies in the following technical tool, used throughout the paper:

\begin{lem}
\label{lem:trick}
Let $\vep > 0$.
In the subsolution (resp. supersolution) condition at the interface, we may assume that
the piecewise quadratic function $P$ touches $u$ strictly and additionally satisfies
\[
F^{\pm}(D^2 P, x_0) < - \vep^{-1} 
\qquad 
\text{(resp. } 
F^{\pm}(D^2 P, x_0) > \vep^{-1}
\text{)}.
\]
\end{lem}
\begin{proof}
We only check the subsolution property, the supersolution one being analogous.
Let $P$ be a piecewise quadratic function touching $u$ from above at $x_0 \in T$.
By translation, we may assume that $x_0 = 0$.

For each small $\delta > 0$, fix $C = C(\delta) > 0$ large to be chosen below.
Consider the modified piecewise quadratic function
\[
\overline{P}(x) 
:= P(x) + \delta |x_n| - \frac{1}{2} C x_n^2 + \frac{1}{2}\delta |x'|^2,
\]
and note that $\overline{P}$ touches $u$ strictly from above at $0$, in a smaller neighborhood.

A computation at $0$ shows that for all $i, j < n$, we have
\[
\overline{P}^{\pm}_{x_i}(0)= P^{\pm}_{x_i}(0), \qquad \overline{P}_{x_n}^{\pm}(0) = P^{\pm}_{x_n}(0) \pm \delta,
\]
\[
\overline{P}_{x_i x_j} = P_{x_i x_j} + \delta \cdot \delta_{ij}, \qquad \overline{P}^{\pm}_{x_i x_n} = P^{\pm}_{x_i x_n}, \qquad \overline{P}^{\pm}_{x_n x_n} = P^{\pm}_{x_n x_n} - C,
\]
whence by uniform ellipticity
\[
F^{\pm}(D^2 \overline{P}^{\pm}, 0) 
\leq F^{\pm}(D^2 P^{\pm}, 0) - \lambda C + \Lambda \delta 
< - \vep^{-1}
\]
by choosing $C > 0$ sufficiently large depending on $\vep$, $\delta$, and $F^{\pm}(D^2 P^{\pm}, 0)$.

If $\overline{P}$ satisfies the transmission condition at $0$, then by uniform ellipticity of $G$ we deduce
\[
g(0) \leq G(\nabla \overline{P}^{+}(0), \theta \nabla \overline{P}^{-}(0), 0) \leq G(\nabla P^{+}(0), \theta \nabla P^{-}(0), 0) + (1+\theta)\delta \Lambda,
\]
and letting $\delta \downarrow 0$ we recover the condition on $P$.
\end{proof}

\begin{rem}
An argument similar to the above shows that the transmission condition is equivalent to the alternative: 
either $G(\nabla u^{+}, \theta \nabla u^{-}, x) = g(x)$ or $F^{\pm}(D^2 u, x) = f^{\pm}(x)$ on $T$.
\end{rem}

\medskip

Our viscosity solutions satisfy the usual closedness property:

\begin{prop}[Closedness]
\label{prop:closed}
Let $\{F_k^{\pm}$, $G_k\}_{k\geq 1}$ be continuous operators on $\Omega^{\pm}\cup T$ and $T$, respectively, with ellipticity constants $0 < \lambda \leq \Lambda$ and $\mu \geq 0$.
Let $\{u_k\}_{k\geq1} \subset \uppers(\Omega)$ be viscosity subsolutions (resp. $\{u_k\}_{k\geq1} \subset \lowers(\Omega)$ be viscosity supersolutions) to 
\[
\begin{cases}
F_k^{\pm}(D^2 u_k, x) = f_k^{\pm}(x) & \text{ in } \Omega^{\pm},\\
G_k(\nabla u^{+}_k, \theta_k \nabla u^{-}_k, x) =  g_k^{\pm}(x) & \text{ on } T,
\end{cases}
\]
for some $\{f_k^{\pm}\} \subset C(\Omega^{\pm} \cup T)$, $\{g_k\} \subset C(T)$, and $\{\theta_k\} \subset [0,1]$.
Assume that $\theta_k \to \theta$ and that
\[
F_{k}^{\pm} \to F^{\pm}, \qquad G_k \to G, \qquad f_{k}^{\pm} \to f^{\pm}, \qquad \text{ and } \qquad g_k \to g
\]
uniformly on compacts in $\scal_n \times (\Omega^{\pm} \cup T)$, $\R^{n}\times \R^{n} \times T$, $\Omega^{\pm} \cup T$, and $T$, respectively.

Then the ``half-relaxed limit''
\[
\textstyle
\limsup^{\star} u_k(x) :=
\displaystyle
\lim_{k \to \infty} \sup \left\{ 
u_l(y) \colon l \geq k, \, y \in B_{1/k}(x)
\right\} 
\in \uppers(\Omega)
\]
\[
\text{(resp. }
\textstyle
\liminf_{\star} u_k(x)
\displaystyle
:= \lim_{k \to \infty} \inf \left\{ 
u_l(y) \colon l \geq k, \, y \in B_{1/k}(x)
\right\} 
\in \lowers(\Omega)
\text{)}
\]
is a viscosity subsolution (resp. supersolution) to~\eqref{eq:nl}.
\end{prop}

\begin{rem}
The half-relaxed limits above are well-known in the viscosity literature and guarantee that the resulting limit function is semicontinuous; for instance, see Section 6 in~\cite{CIL}.
In particular, if a sequence $\{u_k\} \subset C(\Omega)$ of viscosity solutions to the above problems converges uniformly on compacts in $\Omega$, then the limit is a viscosity solution to~\eqref{eq:nl}.
\end{rem}

\begin{proof}
Let $P$ be a piecewise quadratic function touching $u$ from above at $x_0 \in \Omega$.
The case $x_0 \in \Omega^{\pm}$ is already known, hence we assume that $x_0 \in T$.
By Lemma~\ref{lem:trick}, we may assume that
\begin{equation}
\label{eq:strict}
F^{\pm}(D^2 P^{\pm}, x_0) < f^{\pm}(x_0).
\end{equation}

Let $\vep > 0$ and $r >0$ be small to be chosen below.
For $k$ large, we have that $u - u_k < \vep r^2$ in $B_r(x_0) \subset \Omega$, hence the piecewise quadratic function 
\[
P_{\vep, r} := P + 2 \vep |x-x_0|^2 - \vep r^2
\]
satisfies
\[
P_{\vep,r} \geq u_k \quad \text{ on } \partial B_r(x_0), \qquad \text{ and } \qquad P_{\vep, r}(x_0) < u_k(x_0).
\]
It follows that a vertical translation $P_{\vep,r} + c_k$, with $c_k > 0$, touches $u_k$ from above at some $x_{k,r}$ in $B_r(x_0)$.
There are three possibilities:
\begin{itemize}
\item either $x_{k,r} \in B_r^{+}(x_0)$ (resp. $x_{k,r} \in B_r^{-}(x_0)$), and hence
\[
F_{k}^{+} (D^2 P_{\vep, r}, x_{k,r}) \geq f_k^{+} (x_{k,r}) \qquad (\text{resp. } F_{k}^{-} (D^2 P_{\vep, r}, x_{k,r}) \geq f_k^{-} (x_{k,r}));
\]
\item or $x_{x,r} \in T_r(x_0)$, and hence
\[
G_k(\nabla P^{+}_{\vep,r}(x_{k,r}), \theta_k \nabla P^{-}_{\vep,r}(x_{k,r}), x_{k,r}) 
\geq g_k(x_{k,r}).
\]
\end{itemize}

Take a converging subsequence $x_{k_r,r} \to x_0$ with $r \downarrow 0$.
Notice that if for a further subsequence $x_{k_r,r} \in B_r^{+}(x_0)$ (resp. $x_{k_r,r} \in B_r^{-}(x_0)$), then taking the limit
\[
F^{+}(D^2 P^{+} + 4 \vep \id, x_0) \geq f^{+}(x_0) \qquad \text{ (resp. } F^{-}(D^2 P^{-} + 4 \vep \id, x_0) \geq f^{-}(x_0) ),
\]
and letting $\vep \downarrow 0$ contradicts~\eqref{eq:strict}.
Therefore, for a subsequence $x_{k_r,r} \in T_r(x_0)$ and taking the limit as $r \downarrow 0$ we obtain
\[
G(\nabla P^{+}(x_{0}), \theta \nabla P^{-}(x_{0}), x_{0}) \geq g(x_{0}).
\]

The proof of the supersolution property is analogous.
\end{proof}

\medskip

Instead of studying the fully nonlinear problem~\eqref{eq:nl} directly, it is more convenient to consider envelopes of linear problems.
The most celebrated ones are the Pucci extremal operators, defined on symmetric matrices $M \in \scal_n$ by
\[
\mcal^{+}(M; \lambda, \Lambda) = \sup_{A \in \ellmat}  \tr{A M} \qquad \text{ and } \qquad \mcal^{-}(M; \lambda, \Lambda) = \inf_{A \in \ellmat}  \tr{A M}.
\]
Thus, the operators $\mcal^{\pm}(D^2u; \lambda, \Lambda)$ can be interpreted as the envelopes of all linear uniformly elliptic operators in non-divergence form 
\[
\tr{A(x) D^2 u} = \sum_{i, j} a_{ij}(x) u_{x_i x_j},
\]
with coefficient matrices $A(x)$ in $\ellmat$.

\medskip
 
We can similarly define the envelopes of linear boundary conditions for the gradient.
Let $\ellvec$ be the set of \emph{oblique coefficients} given by
\[
\ellvec = 
\left\{
\gamma = (\gamma', \gamma_n) \in \R^{n-1} \times \R \colon \quad \lambda \leq \gamma_n \leq \Lambda, \quad |\gamma'| \leq \mu
\right\}.
\]
For vectors $\xi = (\xi', \xi_n) \in \R^{n-1}\times \R$, we define the upper operator $\ncal^{+}$ by
\begin{equation}
\label{nplus}
\ncal^{+}(\xi; \lambda,\Lambda,\mu)
 = \sup_{\gamma \in \ellvec} \gamma \cdot \xi
= \left( \Lambda \chara{\{\xi_n > 0\}} + \lambda \chara{\{\xi_n < 0\}} \right) \xi_n + \mu |\xi'|
\end{equation}
and the lower operator $\ncal^{-}$ by
\begin{equation}
\label{nminus}
\ncal^{-}(\xi; \lambda,\Lambda,\mu)
= \inf_{\gamma \in \ellvec} \gamma \cdot \xi
= \left( \lambda \chara{\{\xi_n > 0\}} + \Lambda \chara{\{\xi_n < 0\}} \right) \xi_n - \mu |\xi'|.
\end{equation}
Note here that $\ncal^{\pm}(\nabla u; \lambda, \Lambda, \mu)$ are the envelopes of all oblique derivative conditions 
\[
\gamma(x) \cdot \nabla u = \sum_{i}\gamma_i (x) u_{x_i},
\]
with coefficients $\gamma(x)$ in $\ellvec$.

\medskip

It is natural to extend this last notion to transmission conditions:

\begin{defn}[Extremal transmission operators]
For vectors 
$\xi^{\pm} = (\xi'^{\pm}, \xi_n^{\pm}) \in \R^{n-1}\times \R$, 
we define the operators $\tcal^{\pm}$ by
\[
\tcal^{+}(\xi^{+}, \xi^{-}) 
= \tcal^{+}(\xi^{+}, \xi^{-}; \lambda,\Lambda,\mu)
:= \sup_{\gamma^{\pm} \in \ellvec} \left[\gamma^{+} \cdot \xi^{+} - \gamma^{-} \cdot \xi^{-}\right],
\]
\[
\tcal^{-}(\xi^{+}, \xi^{-})  
= \tcal^{-}(\xi^{+}, \xi^{-}; \lambda,\Lambda,\mu)
:= \inf_{\gamma^{\pm} \in \ellvec} \left[\gamma^{+} \cdot \xi^{+} - \gamma^{-} \cdot \xi^{-}\right].
\]
\end{defn}

Note by definition that
\[
\tcal^{\pm}(\xi^{+}, \xi^{-}; \lambda,\Lambda,\mu) = \ncal^{\pm}(\xi^{+}; \lambda, \Lambda, \mu) - \ncal^{\mp}(\xi^{-}; \lambda, \Lambda, \mu),
\]
hence $\tcal^{\pm}$ have explicit expressions given in terms of~\eqref{nplus} and~\eqref{nminus}.
As above, we can view the operators $\tcal^{\pm}(\nabla u^{+}, \nabla u^{-}; \lambda, \Lambda, \mu)$ as envelopes of linear oblique derivative transmission conditions 
\[
\gamma^{+}(x) \nabla u^{+} - \gamma^{-}(x) \cdot \nabla u^{-}
= \sum_{i} \gamma_i^{+}(x) u_{x_i}^{+} 
- \sum_{j} \gamma_j^{-}(x) u_{x_j}^{-},
\]
with coefficients $\gamma^{\pm}(x)$ in $\ellvec$.

\medskip

We note the invariance property
\[
\tcal^{\pm}(- \xi^{+}, - \xi^{-}) = - \tcal^{\mp}(\xi^{+}, \xi^{-}).
\]

\medskip

In the special case when the vectors satisfy $\xi_n^{+} \geq 0$, $\xi_n^{-} \geq 0$, and $\xi'^{+} = \xi'^{-} = \xi'$, we have that
\begin{equation}
\label{tmin}
\tcal^{-}(\xi^{+}, \xi^{-}; \lambda, \Lambda, \mu) = \lambda \xi_n^{+} - \Lambda \xi_n^{-} - 2\mu |\xi'|,
\end{equation}
and considering the interpolation parameter $0 \leq \theta \leq 1$, we see that
\[
\tcal^{-}(\xi^{+}, \theta \xi^{-}) \geq \tcal^{-}(\xi^{+}, \xi^{-})
\]
in this case.
Universal bounds of this type will yield estimates that hold uniformly in $0 \leq \theta \leq 1$.

\medskip

\begin{rem}
If $A = (a_{ij}) \in \ellmat$, then the vector $\gamma = A e_n$ belongs to $\env_{\lambda, \Lambda,(\Lambda-\lambda)/2}$.\footnote{Note that since the set $\ellmat$ is compact and convex, the convex function $A \mapsto |\gamma'| = |A e_n - a_{nn} e_n|$ attains its maximum at the extremal points of $\ellmat$, which are of the form $\lambda P + \Lambda (\id - P)$, with $P \colon \R^n \to \R^n$ an orthogonal projection. It is then easy to find the maximum $\frac{1}{2}(\Lambda-\lambda)$.}
In particular, our extremal operators control the envelopes of all conormal derivatives and their associated transmission conditions.
These are the natural Neumann-type boundary conditions arising in classical variational problems, for which transmission problems are well understood.

The variational theory is based on the equivalence between the divergence-form equation
\begin{equation}
\label{eq:div}
\div (A(x) \nabla u) = 0 \qquad \text{ in } B_1,
\end{equation}
with (say) piecewise constant coefficients $A(x) = A^{+} \chara{B_1^{+}} + \theta A^{-} \chara{B_{1}^{-}}$, and the transmission problem
\begin{equation}
\label{eq:var}
\begin{cases}
\tr{A^{\pm}D^2 u} = 0 & \text{ in } B_1^{\pm},\\
A^{+} \nabla u^{+} \cdot e_n = \theta A^{-} \nabla u^{-} \cdot e_n & \text{ on } T_1.\\
\end{cases}
\end{equation}
An easy way to see this is by formally integrating~\eqref{eq:div} by parts.
More rigorously, $W^{1,2}$ weak solutions to~\eqref{eq:div} are continuous by De Giorgi's technique and hence, by a calibration argument, viscosity solutions to~\eqref{eq:var} (for instance, see our work~\cite{CEF}).
Our regularity theory in the companion paper~\cite{E} shows the converse: that viscosity solutions to~\eqref{eq:var} are also weak solutions to~\eqref{eq:div}.
\end{rem}

\medskip

The extremal operators above lead to universal classes of viscosity subsolutions and supersolutions to transmission problems. 
We introduce the following:

\begin{defn}
We say that $u \in \underline{S}(\lambda,\Lambda,\mu,\theta; f^{+}, f^{-}, g)$\emph{ in }$\Omega$, if $u \in \uppers(\Omega)$ is a viscosity subsolution to the transmission problem
\[
\begin{cases}
\mcal^{+}(D^2 u; \frac{\lambda}{n}, \Lambda) \geq f^{\pm}(x) & \text{ in } \Omega^{\pm},\\
\tcal^{+}(\nabla u^{+}, \theta \nabla u^{-}; \lambda,\Lambda,\mu) \geq g(x) & \text{ on } T.
\end{cases}
\]
Similarly, we write $u \in \overline{S}(\lambda,\Lambda,\mu,\theta; f^{+}, f^{-}, g)$\emph{ in }$\Omega$, if $u \in \lowers(\Omega)$ is a viscosity supersolution to
\[
\begin{cases}
\mcal^{-}(D^2 u; \frac{\lambda}{n}, \Lambda) \leq f^{\pm}(x) & \text{ in } \Omega^{\pm},\\
\tcal^{-}(\nabla u^{+}, \theta \nabla u^{-}; \lambda,\Lambda,\mu) \leq g(x) & \text{ on } T.
\end{cases}
\]
We also define the solution class
\[
S(\lambda,\Lambda,\mu,\theta; f^{+}, f^{-}, g) = \underline{S}(\lambda,\Lambda,\mu,\theta; f^{+}, f^{-}, g) \cap \overline{S}(\lambda,\Lambda,\mu,\theta; f^{+}, f^{-}, g).
\]
\end{defn}

\begin{rem}
By the invariance of extremal operators, the classes $\overline{S}$ and $\underline{S}$ can be recovered from each other by a sign change.
For instance, if $u \in \overline{S}(\lambda, \Lambda, \mu, \theta; f^{+}, f^{-}, g)$ in $\Omega$, then $-u \in \underline{S}(\lambda, \Lambda, \mu, \theta; -f^{+}, -f^{-}, -g)$ in $\Omega$.
\end{rem}

\begin{rem}
Our solutions classes are natural extensions of the classical solution classes for Pucci operators.
In particular, we would like to point out that if $u \in \overline{S}(\lambda, \Lambda, \mu, \theta; f^{+}, f^{-}, g)$ in $\Omega$, 
then, using the notation from Section~2 in \cite{CC}, we have $u^{\pm} \in \overline{S}(\frac{\lambda}{n}, \Lambda; f^{\pm})$ in $\Omega^{\pm}$.
\end{rem}

\medskip

Solution classes behave well under the addition of admissible barriers.

\begin{lem}
\label{lem:aux}
Let $u \in \underline{S}(\lambda,\Lambda,\mu,\theta; f^{+}, f^{-}, g)$ in $\Omega$, and let $\varphi$ be an admissible function satisfying
\[
\begin{cases}
\mcal^{+}(D^2 \varphi; \frac{\lambda}{n}, \Lambda) \leq \widetilde{f}^{\pm}(x) & \text{ in } \Omega^{\pm},\\
\tcal^{+}(\nabla \varphi^{+}, \theta \nabla \varphi^{-}; \lambda, \Lambda, \mu) \leq \widetilde{g}(x) & \text{ on } T.
\end{cases}
\]
Then
\[
u - \varphi \in \underline{S}(\lambda,\Lambda,\mu,\theta; f^{+}- \widetilde{f}^{+}, f^{-}- \widetilde{f}^{-}, g - \widetilde{g}) \quad \text{ in } \Omega.
\]
\end{lem}
\begin{proof}
Let $P$ be a piecewise quadratic function touching $u- \varphi$ from above at $x_0 \in \Omega$.
The case $x_0 \in \Omega^{\pm}$ is known, hence we assume $x_0 \in T$.
Since the piecewise $C^{1,1}$ function $P + \varphi$ touches $u$ from above at $x_0 \in T$, by definition of viscosity subsolution, it follows that
\[
\tcal^{+}( \nabla (P + \varphi)^{+}(x_0), \theta \nabla (P + \varphi)^{-}(x_0)) \geq g(x_0).
\]
On the other hand, by definition of $\tcal^{+}$ and the assumption on $\varphi$, on $T$ we have that
\[
\tcal^{+}(\nabla (P + \varphi)^{+}, \theta \nabla (P + \varphi)^{-}) 
\leq \tcal^{+}(\nabla P^{+}, \theta \nabla P^{-}) + \tcal^{+}(\nabla \varphi^{+},\theta \nabla \varphi^{-}) 
\leq \tcal^{+}(\nabla P^{+},\theta \nabla P^{-}) + \widetilde{g},
\]
which combined with the above yields
\[
\tcal^{+}(\nabla P^{+}(x_0), \theta \nabla P^{-}(x_0)) \geq g (x_0) - \widetilde{g}(x_0).
\]
\end{proof}

\medskip

The next property reduces the analysis of general fully nonlinear problems to the particular case of extremal operators:

\begin{prop}
Let $u \in \uppers(\Omega)$ be a viscosity subsolution to~\eqref{eq:nl}.
Then
\[
u \in \underline{S}\left(
\lambda,\Lambda,\mu,\theta; 
f^{+} - F^{+}(0,\cdot), f^{-} - F^{-}(0,\cdot), g - G(0,0, \cdot)
\right) \quad \text{ in } \Omega.
\]
\end{prop}
\begin{proof}
Let $P$ be a piecewise quadratic function touching $u$ from above at $x_0 \in \Omega$.
The case $x_0 \in \Omega^{\pm}$ is classical, hence we assume that $x_0 \in T$.
Since $u$ is a subsolution, we have that
\[
G(\nabla P^{+}(x_0), \theta \nabla P^{-}(x_0), x_0) \geq g (x_0)
\]
and by uniform ellipticity
\[
G(\nabla P^{+}(x_0), \theta \nabla P^{-}(x_0), x_0) \leq G(0, 0, x_0) + \tcal^{+}(\nabla P^{+}(x_0), \theta \nabla P^{-}(x_0); \lambda, \Lambda, \mu),
\]
whence the claim follows.
\end{proof}

\medskip

We conclude this section by giving a simple characterization of the degenerate problem.
Namely, when $\theta = 0$, our transmission problem decouples into an oblique derivative problem above and a Dirichlet problem below, with boundary data on the interface.
More precisely, we have:

\medskip

\begin{prop}
\label{prop:deg}
The following are equivalent:
\begin{enumerate}[label=\rm(\roman*)]
\item The function $u \in C(\Omega)$ is a viscosity solution to the degenerate transmission problem
\[
\begin{cases}
F^{\pm}(D^2 u, x) = f^{\pm}(x) & \text{ in } \Omega^{\pm},\\
G(\nabla u^{+}, 0, x) = g(x) & \text{ on } T.
\end{cases}
\]
\item The functions $u^{+} \in C(\Omega^{+}\cup T)$ and $u^{-} \in C(\Omega^{-}\cup T)$ are viscosity solutions to the boundary-value problems
\[
\begin{cases}
F^{+}(D^2 u^{+}, x) = f^{+}(x) & \text{ in } \Omega^{+},\\
G(\nabla u^{+},0, x) = g(x) & \text{ on } T,\\
\end{cases}
\quad \text{ and } \quad 
\begin{cases}
F^{-}(D^2 u^{-}, x) = f^{-}(x) & \text{ in } \Omega^{-},\\
u^{-} = u^{+} & \text{ on } T,\\
\end{cases}
\]
respectively.
\end{enumerate}
\end{prop}
\begin{rem}
\label{rem:intuition}
The boundary regularity for the problems in (ii) is well-known: for smooth coefficients and sources, the functions $u^{\pm}$ are $C^{1,\alpha}$ up to the interface.
This suggests that if solutions to pure transmission problems ($\theta = 1$) were piecewise $C^{1,\alpha}$, then one should be able to interpolate between both regularity theories and obtain estimates that hold uniformly in $0 \leq \theta \leq 1$. 
We make this intuition rigorous in our companion paper~\cite{E}, where we obtain piecewise $C^{1,\alpha}$ estimates independent of $\theta$.
\end{rem}

\begin{proof}
That (ii) implies (i), with $u = u^{+}\chara{\Omega^{+}\cup T} + u^{-}\chara{\Omega^{-}} \in C(\Omega)$, is clear.

\medskip

Assume (i) and let us prove that (ii) holds with $u^{+} = u \chara{\Omega^{+}\cup T}$ and $u^{-} = u \chara{\Omega^{-}\cup T}$.
It suffices to show that $u^{+}$ satisfies the oblique derivative condition in the viscosity sense.
We only check the subsolution condition, the supersolution  one being analogous.

\medskip

Let $P$ be a quadratic polynomial touching $u^{+}$ from above at $x_0 \in T$.
For $r > 0$ and $\rho >0$ small, we consider the cylinder $Q_{r,\rho}(x_0) \subset \Omega$ given by
\[
Q_{r, \rho}(x_0) := 
T_r(x_0) \times (-\rho, \rho) = 
\left\{
x = (x', x_n) \in \R^{n-1} \times \R \colon \quad |x'- x'_0| < r, \quad |x_n| < \rho
\right\},
\]
where by a slight abuse of notation we identify $T_r(x_0) = B'_r(x_0') \subset \R^{n-1}$.

\medskip

For $\delta >0$, let $P_{\delta}(x) := P(x) + \frac{1}{2} \delta |x|^2$ and note that $P_{\delta}$ touches $u^{+}$ strictly.
Since $P \geq u$ in a neighborhood of $x_0$ in $\Omega^{+}\cup T$, restricting to a sphere in the interface, we see that
\begin{equation}
\label{abo1}
P_{\delta} \geq u + \frac{1}{2} \delta r^2 \quad \text{ on } \partial T_r(x_0).
\end{equation}
Since $u$ is continuous in $\Omega$, for a small $0 < \rho = \rho(r, \delta) < \delta$ we have that
\begin{equation}
\label{abo2}
|u(x',x_n) - u(x',0)| \leq \frac{1}{4} \delta r^2 \qquad \text{ for } (x', x_n) \in \overline{Q_{r,\rho}}(x_0).
\end{equation}
We extend $P_{\delta}$ to a piecewise quadratic function on the other side by letting
\[
P_{\delta}^{-}(x', x_n) := P_{\delta}(x',0) - \frac{1}{2} \delta r^2 \left[2 \frac{x_n}{\rho} + \frac{x_n^2}{\rho^2}\right] \qquad \text{ for } (x', x_n) \in \overline{Q_{r, \rho}^{-}}(x_0).
\]
Note that
\begin{equation}
\label{abo3}
P^{-}_{\delta}(x',x_n) \geq P_{\delta}(x',0) \quad \text{ for } (x', x_n) \in \overline{Q_{r, \rho}^{-}}(x_0),
\end{equation}
as well as
\begin{equation}
\label{abo4}
P^{-}_{\delta}(x',- \rho) = P_{\delta}(x',0) + \frac{1}{2} \delta r^2 \quad \text{ for } x' \in \partial T_r(x_0).
\end{equation}

We claim that 
\begin{equation}
\label{abo0}
P_{\delta} > u \qquad \text{ on } \partial Q_{r,\rho}.
\end{equation}
Indeed, we check the inequality on the following partition of the boundary:
\begin{itemize}
\item 
On $\partial T_{r}(x_0) \times [-\rho, 0]$, by~\eqref{abo3},~\eqref{abo1}, and~\eqref{abo2}
\[
P_{\delta} - u \geq (P_{\delta} - u)|_{\partial T_{r}(x_0)} - \frac{1}{4} \delta r^2 \geq \frac{1}{4} \delta r^2;
\]
\item On $T_r (x_0) \times \{x_n = - \rho r\}$, by~\eqref{abo4},~\eqref{abo2}, and recalling that $P_{\delta} \geq u$ on $T_{r}(x_0)$
\[
P_{\delta} - u \geq (P_{\delta} - u)|_{T_{r}(x_0)} + \frac{1}{4} \delta r^2 \geq \frac{1}{4} \delta r^2;
\]
\item On $\partial Q_{r,\rho}(x_0) \cap \{x_n > 0\}$, since $P \geq u$ on $\overline{Q_{r,\rho}^{+}}$
\[
P_{\delta} - u \geq \frac{1}{2}\delta r^2.
\]
\end{itemize}

\medskip

A computation shows that
\[
D^2 P_{\delta}^{-} = D^2 [P_{\delta}|_{T}] - \frac{\delta r^2}{\rho^2} e_n \otimes e_n.
\]
Choosing $\rho > 0$ smaller, we may assume that
\[
\mcal^{+}(D^2 P_{\delta}^{-}; \textstyle \frac{\lambda}{n}, \Lambda)
\leq \displaystyle- \frac{\lambda}{n}\frac{\delta r^2}{\rho^2} + \Lambda(n-1) \|D^2 [P_{\delta}|_{T}]\|
\leq - \|f^{-} - F^{-}(0,\cdot)\|_{L^{\infty}(Q_{r,\rho}^{-})},
\]
and hence by uniform ellipticity
\begin{equation}
\label{abo5}
F^{-}(D^2 P_{\delta}^{-}, x) \leq F^{-}(0, x) + \mcal^{+}(D^2 P_{\delta}^{-}; \textstyle \frac{\lambda}{n}, \Lambda) < f^{-}(x) \qquad \text{ for } x \in Q_{r,\rho}^{-}.
\end{equation}

By~\eqref{abo0}, a vertical translation $P_{\delta} + c$, with $c > 0$, touches $u$ from above at some $x_r \in Q_{r,\rho}(x_0)$.
Moreover $x_r \in T_r(x_0)$ since otherwise the bulk equation would contradict~\eqref{abo5}.
The transmission condition at $x_r$ now gives
\[
g(x_r) \leq G(\nabla P_{\delta}^{+}(x_r), 0, x_r) \leq G(\nabla P^{+}(x_r), 0, x_r) + \delta \mu r,
\]
where we have used that $\nabla P_{\delta}^{+}(x',0) = \nabla P_{\delta}(x',0) + \delta (x',0)$ combined with uniform ellipticity.
Letting $r \downarrow 0$ yields the claim.
\end{proof}


\medskip

\section{The ABP maximum principle}
\label{sec:abp}

The classical ABP maximum principle provides a quantitative lower bound for viscosity supersolutions to second order PDEs.
In this section, we prove a variant of this result for oblique transmission problems.
For this, we touch the supersolution from below by appropriate piecewise linear functions.
Choosing the coefficients on each side to be comparable in size, the usual ABP technique will allow us to control the original function by the volume of its gradient image.

\medskip

In the next statement, recall our subindex notation for the positive and negative parts
\[
u_{+} = \max\{u, 0\}
\qquad \text{ and }
\qquad 
u_{-} = \max\{- u, 0\},
\]
and compare it with the superindex notation for the restrictions
\[
u^{+} = u \chara{\{x_n \geq 0\}} \qquad \text{ and } \qquad u^{-} = u \chara{\{x_n \leq 0\}}.
\]

\medskip

\begin{thm}[ABP estimate]
\label{thm:abp}
Let $u \in \overline{S}(\lambda,\Lambda,\mu, \theta; f^{+}, f^{-}, g)$ in $B_R$, with $u \in \lowers(\overline{B_R})$.

There is a universal constant $C > 0$ such that
\[
\sup_{B_R} u_{-} \leq \sup_{\partial B_R} u_{-} + C R \left( 
 \|f^{+}_{+}\|_{L^{n}(\{u = \Gamma_u\} \cap B_R^{+})} + \|f^{-}_{+}\|_{L^{n}(\{u = \Gamma_u\} \cap B_R^{-})}+ \|g_{+}\|_{L^{\infty}(T_R)}
 \right),
\]
where $\Gamma_u$ denotes the convex envelope of $- u_{-} \chara{B_R}$ in the ball $B_{2R}$.
\end{thm}

\begin{proof}
Consider the piecewise linear barrier
\begin{equation}
\label{abp:bar}
\varphi(x) := \frac{1}{\lambda} \|g_{+}\|_{L^{\infty}(T_R)} (x_n)_{+}.
\end{equation}
Since $\varphi_{x_n}^{+} \geq 0$ and $\varphi_{x_n}^{-} = 0$, from~\eqref{tmin}, we have that
\[
\tcal^{-}(\nabla \varphi^{+},\theta \nabla \varphi^{-}; \lambda, \Lambda, \mu)
= \lambda \varphi_{x_n}^{+}
=  \|g_{+}\|_{L^{\infty}(T_R)},
\]
and hence by Lemma~\ref{lem:aux}
\begin{equation}
\label{abp:sup}
u- \varphi \in \overline{S}(\lambda, \Lambda, \mu, \theta; f^{+}, f^{-}, 0) \quad \text{ in } B_R.
\end{equation}

\medskip

For $\xi = (\xi', \xi_n) \in \R^{n-1} \times (0,\infty)$, let $\ell_{\xi}$ be the continuous piecewise linear function
\begin{equation}
\label{abp:piece}
\ell_{\xi} 
= \xi' \cdot x' + \xi_n\Big\{ 2\frac{\Lambda}{\lambda} (x_n)_{+} - (x_n)_{-}\Big\}.
\end{equation}
Since $(\ell_{\xi})_{x_n}^{+} = 2 \frac{\Lambda}{\lambda}\xi_n >0$ and $(\ell_{\xi})_{x_n}^{-} = \xi_n >0$, recalling $0 \leq \theta \leq 1$, from~\eqref{tmin} we have that
\[
\tcal^{-} (\nabla \ell_{\xi}^{+}, \theta \nabla \ell_{\xi}^{-}; \lambda, \Lambda, \mu) 
= \lambda(\ell_{\xi})_{x_n}^{+} - \Lambda (\ell_{\xi})_{x_n}^{-} - \mu(1+\theta) |\nabla' \ell_{\xi}|
\geq \Lambda \xi_n - 2 \mu |\xi'|.
\]
Therefore, if a vertical translation of $\ell_{\xi}$ touches $u-\varphi$ from below somewhere at $T_R$, then 
\begin{equation}
\label{abp:trans}
\Lambda \xi_n - 2\mu |\xi'| \leq 0
\end{equation}
by definition of viscosity supersolution.

\medskip

Let $v$ be the auxiliary function
\begin{equation}
\label{abp:v}
v := u- \varphi + \sup_{\partial B_R} (u-\varphi)_{-}
\end{equation}
and observe that $v \geq 0$ on $\partial B_R$.
We assume that its negative part is nontrivial, with
\begin{equation}
\label{abp:mdef}
M := \sup_{B_R} v_{-} > 0.
\end{equation}
Let $c_{\xi} \in \R$ be such that the vertical translation $\widetilde{\ell}_{\xi} := \ell_{\xi} + c_{\xi}$ touches $- v_{-} \chara{B_R}$ from below on $\overline{B}_{2R}$.
Finally, recall that $\Gamma_{v}$ is the convex envelope of $- v_{-} \chara{B_R}$ in $B_{2R}$, i.e., by definition
\[
\Gamma_v(x) = \sup\big\{
\ell(x) \colon \quad \ell \text{ affine}, \quad \ell \leq - v_{-} \chara{B_R} \, \text{ in } B_{2R}
\big\}.
\]

\medskip

We note the following properties of $\widetilde{\ell}_{\xi}$ depending on the dimensions of the parameter $\xi$:
\begin{enumerate}[label=(\roman*)]
\item If $|\xi| < \frac{\lambda}{6 \Lambda} \frac{M}{R}$, then $|\nabla \ell_{\xi}^{\pm}| < \frac{M}{3R}$ and hence $\widetilde{\ell}_{\xi}$ touches $-v_{-} \chara{B_R}$ from below in $B_R$.
\item If $\xi_n > \frac{2 \mu}{\Lambda} |\xi'|$, then by~\eqref{abp:trans} and the discussion above, the function $\widetilde{\ell}_{\xi}$ touches $-v_{-}\chara{B_R}$ from below in $\overline{B}_{2R} \setminus T_R$, i.e., away from the interface~$T_R$.
\end{enumerate}
Motivated by conditions (i) and (ii), we define the coefficient set 
\[
\Xi := \left\{\xi = (\xi', \xi_n) \in \R^{n-1}\times \R \colon \quad  \xi_n > \frac{2\mu}{\Lambda} |\xi'|, \quad |\xi| < \frac{\lambda}{6 \Lambda}\right\}.
\]
By construction, each $\xi \in \frac{M}{R}\Xi$ yields a piecewise linear function $\widetilde{\ell}_{\xi}$ below $- v_{-} \chara{B_R}$ in $\overline{B}_{2R}$ and touching $-v_{-}$ in $B_R\setminus T_R$.
Therefore, recalling definition~\eqref{abp:piece}, we have the following alternative depending on whether such a tangency point lies in $B_R^{-}$ or in $B_R^{+}$:
\[
\nabla \widetilde{\ell}^{-}_{\xi} = \xi 
\in \nabla \Gamma_{v}(B_R \setminus T_R)
\qquad \text{ or } \qquad 
\nabla \widetilde{\ell}^{+}_{\xi} = A (\xi) \in \nabla \Gamma_{v}(B_R \setminus T_R),
\]
where $A \colon \R^{n} \to \R^{n}$ is the linear map given by
\[
A(\xi', \xi_n) = (\xi', \textstyle\frac{2\Lambda}{\lambda} \xi_n) \qquad \text{ for } (\xi', \xi_n) \in \R^{n-1} \times \R.
\]

\medskip

We partition the rescaled coefficients by $\frac{M}{R}\Xi = X_1 \cup X_2$, where
\[
X_1 := \textstyle\frac{M}{R}\Xi
\cap \nabla \Gamma(B_R \setminus T_R)
\qquad \text{ and } \qquad 
X_2 := \textstyle \frac{M}{R}\Xi \setminus X_1.
\]
From the dichotomy above, it follows that
\[
X_1 \cup A(X_2) \subset \nabla \Gamma_{v}(B_R \setminus T_R)
\]
and taking the Lebesgue measure on both sides, we obtain the lower bound
\begin{equation}
\label{eq:trinity}
|\nabla \Gamma_{v}(B_R \setminus T_R)| \geq \max\left\{|X_1|, |A(X_2)| \right\} \geq \max\left\{|X_1|, |X_2| \right\} \geq \frac{1}{2}\left|\frac{M}{R}\Xi\right| = \frac{1}{2} |\Xi| \left(\frac{M}{R}\right)^{n}.
\end{equation}
On the other hand, since $\Gamma_v$ is $C^{1,1}$ away from the interface and recalling~\eqref{abp:sup} and~\eqref{abp:v}, by the classical argument we deduce an upper bound
\begin{equation}
\label{eq:trin2}
|\nabla \Gamma_{v}(B_R \setminus T_R)|
\leq \int_{\{v = \Gamma_v\}\cap (B_R\setminus T_R)} D^2 \Gamma_{v}
\leq C \left( \int_{\{u = \Gamma_u\}\cap B_R^{+}} (f_{+}^{+})^{n} 
+ \int_{\{u = \Gamma_u\}\cap B_R^{-}} (f_{+}^{-})^{n}
\right),
\end{equation}
where in the last line we used that $\{v = \Gamma_v\} \subset \{u = \Gamma_u\}$, a fact which follows by convexity of $\varphi$.
Since $\Xi$ has universal positive measure,~\eqref{eq:trinity} and~\eqref{eq:trin2} yield
\begin{equation}
\label{abp:puerta}
M \leq C R \left(\|f_{+}^{+}\|_{L^{n}(\{u = \Gamma_{u}\}\cap B_R^{+} )} + \|f_{+}^{-}\|_{L^{n}(\{u = \Gamma_{u}\}\cap B_R^{-} )}\right).
\end{equation}

\medskip

Combining~\eqref{abp:bar},~\eqref{abp:v},~\eqref{abp:mdef}, and~\eqref{abp:puerta}, we finally deduce
\[
\begin{split}
\sup_{B_R} u_{-}
&\leq \sup_{\partial B_R} (u-\varphi)_{-} + \sup_{B_R} v_{-}\\
&\leq \sup_{\partial B_R} u_{-}+ C R \|g_{+}\|_{L^{\infty}(T_R)} + C R \left(  
\|f_{+}^{+}\|_{L^{n}(\{u = \Gamma_{u}\}\cap B_R^{+} )} + \|f_{+}^{-}\|_{L^{n}(\{u = \Gamma_{u}\}\cap B_R^{-} )}
\right),
\end{split}
\]
which was the claim.
\end{proof}

\medskip

The ABP maximum principle admits an extension to bounded domains:

\begin{cor}
Let $\Omega \subset B_R \subset \R^n$ be a bounded open set and let $u \in \overline{S}(\lambda,\Lambda,\mu, \theta; f^{+}, f^{-}, g)$ in $\Omega$, with $u \in \lowers(\overline{\Omega})$, $f^{\pm} \in C(\Omega^{\pm}\cup T) \cap L^{\infty}(\Omega)$, and $g \in C(T) \cap L^{\infty}(T)$.

There is a universal constant $C > 0$ such that
\[
\sup_{\Omega} u_{-} \leq \sup_{\partial \Omega} u_{-} + 
C R \left( 
\|f_{+}^{+}\|_{L^{n}(\{u = \Gamma_u\} \cap \Omega^{+})} 
+ \|f_{+}^{-}\|_{L^{n}(\{u = \Gamma_u\} \cap \Omega^{-})}
+\|g_{+}\|_{L^{\infty}(T)}
\right),
\]
where $\Gamma_u$ denotes the convex envelope of $- u_{-} \chara{\Omega}$ in the ball $B_{2 R}$.
\end{cor}
\begin{proof}
It is not difficult to see that
$- u_{-} + \sup_{\partial \Omega} u_{-} \in \overline{S}(\lambda,\Lambda,\mu,\theta; f^{+}_{+}\chara{\Omega^{+}},f^{-}_{+} \chara{\Omega^{-}},g_{+} \chara{T})$ in $B_R$.
Using the piecewise linear barrier $\varphi = \frac{1}{\lambda} \|g_{+}\|_{L^{\infty}(T)}$ and following the proof of Theorem~\ref{thm:abp}, we see that the contact points in $B_{R}$ belong to $\Omega^{+} \cup \Omega^{-}$.
The result follows.
\end{proof}

\begin{rem}
By the invariance property of the extremal operators, we obtain analogous statements for subsolutions $u \in \underline{S}(\lambda, \Lambda, \mu, \theta; f^{+}, f^{-}, g)$ in $\Omega$, with $u \in \uppers(\overline{\Omega})$.
\end{rem}


\medskip

\section{Comparison principle and applications}
\label{sec:unique}

Here, we consider the constant-coefficient inhomogeneous transmission problem in cylinders
\begin{equation}
\label{eq:const}
\begin{cases}
F^{\pm}(D^2 u) = f^{\pm}(x) & \text{ in } Q_R^{\pm},\\
G(\nabla u^{+}, \theta \nabla u^{-}) = g(x) &\text{ on } T_R,
\end{cases}
\end{equation}
where $Q_R \subset \R^{n}$ is the standard cylinder
\[
Q_{R} = B_R' \times (-R, R) \subset \R^{n-1} \times \R
\]
and in particular
\[
Q_{R}^{+} = B_R' \times (0, R), \qquad Q_{R}^{-} = B_R' \times (-R, 0), \qquad T_R = B_R' \times \{0\}.
\]

\medskip

Our main goal is to establish a comparison principle for viscosity solutions to~\eqref{eq:const}.
Such a result is fundamental to both the existence and uniqueness theory for the associated Dirichlet problem, as well as the optimal piecewise $C^{1,\alpha}$ regularity in~\cite{E}. 
Our approach is adapted from~\cite[Section~3]{DFSfb,DFS} and~\cite[Section~4]{SS}, where the authors studied linear Neumann transmission conditions of the form $G = \xi^{+}_n - \xi^{-}_{n}$.
The proof in the nonlinear oblique case~\eqref{eq:const} follows the same strategy, with some modifications (see especially Remark~\ref{rem:mistake} after the proof of Theorem~\ref{thm:comp}).

\medskip

Let $\rho < R$ and $u \in \uppers(Q_R)$.
For $\vep > 0$, we define the upper envelope of $u$ in the horizontal direction by
\[
u^{\vep}(x', x_n) = \sup_{y' \in B_\rho'}\left[u(y', x_n) - \frac{1}{\vep} |y' -x'|^2\right] \qquad \text{ for } x= (x',x_n) \in Q_{\rho}.
\]
This is the classical sup-convolution of Jensen applied to the trace $u(\cdot, x_n)$, for each $|x_n| < \rho$.
Such a regularization of $u$ has nice properties:

\begin{lem}[c.f.~Lemma~3.1 in~\cite{DFS} and Proposition~4.5 in~\cite{SS}]
\label{lem:envelopes}
The following properties hold:
\begin{enumerate}[label=\rm{(\roman*)}]
\item 
$u^{\vep} \in \uppers(\overline{Q}_{\rho})$, $u^{\vep} \downarrow u$ pointwise in $Q_{\rho}$ as $\vep \downarrow 0$, and hence for any sequence $\vep_k \downarrow 0$
\[
u(x) = \textstyle\limsup^{\star} u^{\vep_k}(x) 
\displaystyle = \lim_{k \to \infty} \sup_{B_{1/k}(x)} u^{\vep_k} \quad \text{ for } x \in Q_{\rho};
\footnotemark\label{foot:half}
\footnotetext{In general, if $\{u_k\}$ is a monotone decreasing sequence of bounded and \usc functions in a set $E \subset \R^n$, then the pointwise limit $u$ must be the half-relaxed limit $u(x) = \limsup^{\star} u_k(x) = \lim_{k \to \infty} \sup_{B_{1/k}(x) \cap E} u_k(x)$. To see this, let $x \in E$ and $\vep >0$. By pointwise convergence $u_k(x) \leq u(x) + \vep$ for $k \geq k_0 = k_0(x,\vep)$, and by upper semicontinuity of $u_{k_0}$ we have that $\sup_{B_{\delta}(x) \cap E} u_{k_0} \leq u_{k_0}(x) + \vep$, where $\delta = \delta(x, \vep) > 0$. Finally, by monotonicity \[ \textstyle \limsup^{\star} u_k(x) = \displaystyle \lim_{k \to \infty }\sup_{B_{1/k}(x) \cap E} u_k \leq \sup_{B_{\delta}(x) \cap E} u_{k_0} \leq u_{k_0}(x) + \vep \leq u(x)+ 2\vep \] and letting $\vep \downarrow 0$ we deduce $\limsup^{\star} u_k \leq u$ in $E$. The reverse inequality is trivial.}
\]
\item 
For each $|x_n| < \rho$, the trace $u^{\vep}(\cdot, x_n)$ is Lipschitz continuous on $\overline{B}_{\rho}' \subset \R^{n-1}$ and pointwise second order differentiable a.e. in $B_{\rho}' \subset \R^{n-1}$;\footnote{That is, for each $|x_n| < \rho$ and almost every $x' \in B'_{\rho}$, there is a quadratic polynomial $P$ such that $u^{\vep}(y',x_n) = P(y') + o(|y'-x'|^2)$ as $y' \to x'$.}
\item If $u$ is a viscosity subsolution to~\eqref{eq:const} in $Q_{R}$, then for $r \leq \rho- r_{\vep}$, with $r_{\vep} := \left(2\vep \|u\|_{L^{\infty}(B_{\rho})}\right)^{1/2}$, the function $u^{\vep}$ is a viscosity subsolution to
\[
\begin{cases}
F^{\pm}(D^2 u^{\vep}) = f^{\pm}(x) - \omega_{f^{\pm}}(r_{\vep}) & \text{ in } Q_{r}^{\pm},\\
G(\nabla (u^{\vep})^{+}, \theta \nabla (u^{\vep})^{-}) = g(x) - \omega_{g}(r_{\vep}) &\text{ on } T_r,
\end{cases}
\]
where $\omega_{f^{\pm}}$ and $\omega_{g}$ are the moduli of continuity on $Q_{\rho}$ given by
\[
\omega_{f^{\pm}}(r) = \sup_{x, y \in Q^{\pm}_{\rho}, \, |x - y| <r} 
|f^{\pm}(x) - f^{\pm}(y)|,
\qquad
\omega_{g}(r) = \sup_{x, y \in T_{\rho}, \, |x - y| <r} 
|g(x) - g(y)|.
\]
\end{enumerate}
\end{lem}

\begin{rem}
An analogous statement holds for $u \in \lowers(Q_R)$ and its lower envelopes
\[
u_{\vep}(x', x_n) 
= \inf_{y' \in B_\rho'}\left[u(y', x_n) + \frac{1}{\vep} |y' -x'|^2\right]
\qquad \text{ for } x= (x',x_n) \in Q_{\rho}.
\]
\end{rem}

\medskip

The proof of Lemma~\ref{lem:envelopes} turns out to be essentially the same as for linear Neumann transmission conditions:

\begin{proof}
The regularity properties (i) and (ii) are classical.
We only check the subsolution property (iii) for the transmission condition, the one for the bulk equations being analogous.

\medskip

Let $P$ be a piecewise quadratic function touching $u^{\vep}$ from above at $x_0 \in T_{r}$.
Since $u^{\vep}(x_0) = u(x_{\vep}) - \vep^{-1}|x' -x'_{\vep}|^2$ for some $x_{\vep} \in \overline{T}_{\rho}$ with $|x_{\vep} - x_{0}| \leq r_{\vep}$, it follows that the translated piecewise quadratic function $Q(y) := P(y + x_0-x_{\vep}) + \vep^{-1} |x' -x'_{\vep}|^2$ touches $u$ at $x_{\vep}$.
By the transmission condition and uniform ellipticity, we then have that
\[
\begin{split}
G(\nabla P^{+}(x_{0}),\theta \nabla P^{-}(x_{0})) = 
G(\nabla Q^{+}(x_{\vep}),\theta \nabla Q^{-}(x_{\vep})) \geq g(x_{\vep}) \geq g(x_0) - \omega_{g}(r_{\vep}),
\end{split}
\]
whence the claim follows.
\end{proof}

\medskip

The comparison principle states that the difference of viscosity sub- and supersolutions to~\eqref{eq:const} belongs to the subsolution class $\underline{S}$.
More precisely, we have the following:

\begin{thm}[Comparison principle -- c.f. Theorem~4.7 in~\cite{SS}]
\label{thm:comp}
Let $\{F^{\pm},G\}$ be constant-coefficient operators with ellipticity constants $0 < \lambda \leq \Lambda$ and $\mu \geq 0$.
Let $\{f^{\pm}, \widetilde{f}^{\pm}\}$ be functions in $C(Q^{\pm}_R \cup T_R)$, let $\{g, \widetilde{g}\}$ be functions in $C(T_R)$, and let $0 \leq \theta \leq 1$ be a constant.

If $u \in \uppers(Q_R)$ and $v \in \lowers(Q_R)$ satisfy
\[
\begin{cases}
F^{\pm}(D^2 u) \geq f^{\pm}(x) & \text{ in } Q_R^{\pm},\\
G(\nabla u^{+}, \theta \nabla u^{-}) \geq g(x) & \text{ on } T_R,
\end{cases}
\quad \text{ and } \quad
\begin{cases}
F^{\pm}(D^2 v) \leq \widetilde{f}^{\pm}(x) & \text{ in } Q_R^{\pm},\\
G(\nabla v^{+}, \theta \nabla v^{-}) \leq \widetilde{g}(x) & \text{ on } T_R,
\end{cases}
\]
respectively, in the viscosity sense, then
\[
u - v \in \underline{S}(\lambda,\Lambda,\mu,\theta; f^{+} - \widetilde{f}^{+}, f^{-} - \widetilde{f}^{-}, g - \widetilde{g}) \quad \text{ in } Q_R.
\]
\end{thm}
\begin{proof}
By Lemma~\ref{lem:envelopes} and the closedness property of viscosity subsolutions (Proposition~\ref{prop:closed}), it suffices to prove that for $r \leq \rho - \max\{r_{\vep}, \widetilde{r}_{\vep}\}$, with $r_{\vep} = (2 \vep \|u\|_{L^{\infty}(B_{\rho})})^{1/2}$ and $\widetilde{r}_{\vep} = (2 \vep \|v\|_{L^{\infty}(B_{\rho})})^{1/2}$, we have that
\[
u^{\vep} - v_{\vep} \in 
\underline{S}(
\lambda, \Lambda, \mu, \theta;
f^{+} - c_{\vep} - (\widetilde{f}^{+} + \widetilde{c}_{\vep}),
f^{-} - c_{\vep} - (\widetilde{f}^{-} + \widetilde{c}_{\vep}),
g- c_{\vep} - (\widetilde{g} + \widetilde{c}_{\vep})
\big) 
\quad \text{ in } Q_{r},
\]
where $c_{\vep}= \max\{\omega_{f^{+}}(r_{\vep}), \omega_{f^{-}}(r_{\vep}), \omega_{g}(r_{\vep})\}$ and $\widetilde{c}_{\vep}= \max\{\omega_{\widetilde{f}^{+}}(\widetilde{r}_{\vep}), \omega_{\widetilde{f}^{-}}(\widetilde{r}_{\vep}), \omega_{g}(\widetilde{r}_{\vep})\}$.

\medskip

Let $P$ be a piecewise quadratic function touching $w := u^{\vep} - v_{\vep}$ from above at $x_0 \in Q_r$.
The case $x_0 \notin T_r$ is classical, hence we assume that $x_0 \in T_r$.
By Lemma~\ref{lem:trick}, we may assume that $P$ touches $w$ strictly and that
\begin{equation}
\label{simple:as}
\textstyle
\mcal^{+}(D^2 P^{\pm}; \frac{\lambda}{n}, \Lambda) < 
- \|f^{\pm}\|_{L^{\infty}(Q_{\rho}^{\pm})}
- c_{\vep}
- \|\widetilde{f}^{\pm}\|_{L^{\infty}(Q_{\rho}^{\pm})}
- \widetilde{c}_{\vep}.
\end{equation}
Since it touches strictly, we have $\eta = \eta(\delta) := \inf_{\partial Q_{\delta}(x_0)} (P - w) > 0$ for all small $\delta >0$.

\medskip

Let $\psi = P - w- \frac{\eta}{2}$, and note that $\psi \geq \frac{\eta}{2} >0$ on $\partial B_{\delta}(x_0)$ and $\psi(x_0) = - \frac{\eta}{2} <0$.
By the ABP method, the set of points in $T_{\delta}(x_0)$ where the trace $\psi|_{T}$ can be touched from below by linear polynomials of arbitrarily small slope has positive measure.
Hence, by (ii) in~Lemma~\ref{lem:envelopes}, for each small $\vartheta > 0$, we can find a linear polynomial $\ell_{\vartheta} = a'_{\vartheta} \cdot x' + c_{\vartheta}$, with $|a'_{\vartheta}| < \vartheta$, touching $\psi|_{T}$ from below at a point $y_{0} \in T_{\delta}(x_0)$ where the traces $u^{\vep}|_{T}$ and $v_{\vep}|_{T}$ are both pointwise twice differentiable.
Taking $\vartheta$ smaller if necessary, we may assume that $\overline{P} := P - \ell_{\vartheta} - \frac{\eta}{2} \geq w$ on $\partial B_{\delta}(x_{0})^{+} \cup \partial B_{\delta}(x_0)^{-}$.
By the classical comparison principle $\mcal^{+}(D^2 w; \frac{\lambda}{n}, \Lambda)\geq  f^{\pm} - c_{\vep} - (\widetilde{f}^{\pm}+\widetilde{c}_{\vep})$ in $Q_{r}^{\pm}$, in the viscosity sense, hence by~\eqref{simple:as} and the classical ABP on each half-ball $B^{\pm}_{\delta}(x_0)$, we deduce
\begin{equation}
\label{imp:step}
\overline{P} \geq w \qquad \text{ on } 
\overline{B}_{\delta}(x_0).
\end{equation}

\medskip

The functions $u^{\vep}$ and $v_{\vep}$ need not be piecewise differentiable at $y_{0}$ a priori, hence we need to consider appropriate replacements.

\medskip

First we regularize the boundary datum $u^{\vep}|_{\partial B^{+}_{\delta}(x_0) \cup \partial B^{-}_{\delta}(x_0)}$ in the normal direction by a further sup-convolution.
For $x =(x',x_n) \in \partial B^{+}_{\delta}(x_0) \cup \partial B^{-}_{\delta}(x_0)$ and $k \geq 0$, let
\[
u_k(x) := \sup_{y \in \partial B^{+}_{\delta}(x_0) \cup \partial B^{-}_{\delta}(x_0)}\Big[u^{\vep}(y) - k |x - y|\Big].
\]
By the usual properties, we have that $u_k \in C(\partial B^{+}_{\delta}(x_0) \cup \partial B^{-}_{\delta}(x_0))$ and
\begin{equation}
\label{important1}
u_k \downarrow u^{\vep} \qquad \text{ on } \partial B^{+}_{\delta}(x_0) \cup \partial B^{-}_{\delta}(x_0) \text{ as } k \to \infty.
\end{equation}
Moreover, we claim that $u_k = u^{\vep}$ on the interface, for large $k$.
More precisely, recalling that $u^{\vep}$ is Lipschitz continuous in the tangential variable (see Lemma~\ref{lem:envelopes}~(ii)), for $k \geq [u^{\vep}]_{C^{0,1}(\overline{T}_{\delta}(x_0))}$ and $\delta' \leq \delta - \delta_{k}$, with $\delta_{k} = 2 \|u^{\vep}\|_{L^{\infty}(B_{\delta}(x_0))} k^{-1}$, we have that
\begin{equation}
\label{important2}
u_k = u^{\vep} \qquad \text{ on } T_{\delta'}(x_0).
\end{equation}

\medskip

To show~\eqref{important2}, first note that $u_{k}(x) = u^{\vep}(x_k) - k |x - x_k|$ for some $x_k \in \partial B^{+}_{\delta}(x_0) \cup \partial B^{-}_{\delta}(x_0)$ with $|x-x_k| \leq \delta_k$. 
Hence, for $x \in T_{\delta'}(x_0)$, we have $x_k \in T_{\delta}(x_0)$ and by Lipschitz continuity of $u^{\vep}|_{T}$
\[
u_k(x) 
= u^{\vep}(x_k) - k |x- x_k| 
\leq u^{\vep}(x)- (k -  [u^{\vep}]_{C^{0,1}(\overline{T}_{\delta}(x_0))}) |x- x_k| \leq u^{\vep}(x),
\]
whence~\eqref{important2} follows.

\medskip

Let $\overline{u}_k \in C(\overline{B}_{\delta}(x_0))$ be defined on each half-ball $\overline{B}_{\delta}^{\pm}(x_0)$ as the unique viscosity solution to the boundary-value problem
\begin{equation}
\label{ma:bdy}
\begin{cases}
F^{\pm}(D^2 \overline{u}_k) = f^{\pm}(x) - c_{\vep}& \text{ in } B_{\delta}^{\pm}(x_0),\\
\overline{u}_k = u_k & \text{ on } \partial B_{\delta}^{\pm}(x_0).
\end{cases}
\end{equation}
By standard comparison $\overline{u}_k \geq u^{\vep}$ on $\overline{B}_{\delta}^{\pm}(x_0)$ and since $\overline{u}_k = u^{\vep}$ on $\overline{T}_{\delta'}(x_0)$, it follows that $\overline{u}_k$ are viscosity subsolutions to the transmission condition on $\overline{T}_{\delta'}(x_0)$, with $\delta'$ and $k$ as in~\eqref{important2}.

\medskip

From~\eqref{important1} and~\eqref{important2}, by comparison we deduce that $\overline{u}_k$ is monotone decreasing on $\overline{B}_{\delta}(x_0)$, with a limit $\overline{u} := \inf_{k} \overline{u}_k \in \uppers(\overline{B}_{\delta}(x_0))$ satisfying
\begin{equation}
\label{limprop}
\overline{u} \geq u^{\vep} \quad \text{ on } \overline{B}_\delta(x_0),
\qquad
\qquad \overline{u} = u^{\vep} \quad \text{ on } \partial B_{\delta}^{+}(x_0) \cup \partial B_{\delta}^{-}(x_0).
\end{equation}
Furthermore, since $\overline{u} = \limsup^{\star} u_k$ in $\overline{B}_{\delta}(x_0)$ (see footnote~\ref{foot:half}), and recalling that $\overline{u}_k$ are subsolutions to the transmission problem, the closedness property (Proposition~\ref{prop:closed}) yields
\begin{equation}
\label{bvp1}
\begin{cases}
F^{\pm}(D^2 \overline{u}) \geq f^{\pm}(x) - c_{\vep} & \text{ in } B_{\delta}^{\pm}(x_0),\\
G(\nabla\overline{u}^{+}, \theta \nabla \overline{u}^{-}) \geq g(x) - c_{\vep} & \text{ on } T_{\delta}(x_0),
\end{cases}
\end{equation}
in the viscosity sense.

\medskip

Since $y_0 \in T_{\delta}(x_0)$, choosing $0 < \delta'' < \delta' \leq \delta - \delta_k$ with $T_{\delta''}(y_0) \subset T_{\delta'}(x_0)$, from~\eqref{important2} we see that
\[
\overline{u}_k = u^{\vep} \qquad \text{ on } T_{\delta''}(y_0).
\]
Therefore, recalling that $u^{\vep}|_{T}$ is pointwise second order differentiable at $y_0$, the interface data $\overline{u}_{k}|_{T_{\delta''}(y_0)} = u^{\vep}$ in~\eqref{ma:bdy} is (a fortiori) $C^{1,\alpha}$ at $y_0$, and by Ma and~Wang's boundary regularity result~\cite{MW} applied to each problem in~\eqref{ma:bdy}, we deduce that the replacements $\overline{u}_k$ are piecewise $C^{1,\alpha}$ at $y_{0}$.
More precisely, there are piecewise linear functions $L_k$ such that for all $0 < r \leq r_{\vep}$
\[
|\nabla L_k^{\pm}| \leq C_{\vep} r,
\]
\[
|\overline{u}_k - L_{k}| \leq C_{\vep} r^{1+\alpha} \quad \text{ in } B_r(y_0),
\]
where $r_{\vep} > 0$ and $C_{\vep} >0$ are  independent of $k$.
Passing to the limit as $k \to \infty$, we obtain a piecewise linear function $L_{\overline{u}} := \lim_{k} L_k$ satisfying
\begin{equation}
\label{lprop}
|\overline{u} - L_{\overline{u}}| \leq C_{\vep} r^{1+\alpha} \quad \text{ in } B_r(y_0).
\end{equation}
One can also show (see the proof of Lemma~4.4 in~\cite{DFS}) that the piecewise quadratic barrier
\[
L_{\overline{u}}(x) + C_{\vep} r^{\alpha} |x_n| + C_{\vep} r^{\alpha-1} |x'- y_{0}'|^2,
\]
with $C_{\vep} >0$ possibly larger, touches $\overline{u}$ from above at $y_{0}$ in $B_{r/2}(y_{0})$.
Therefore, by the subsolution property at $y_0$ in~\eqref{bvp1} and letting $r \downarrow 0$, we finally deduce
\begin{equation}
\label{eq:key2}
G(\nabla L_{\overline{u}}^{+}, \theta \nabla L_{\overline{u}}^{-}) \geq g(y_0) - c_{\vep}.
\end{equation}

\medskip
 
Arguing similarly with $v_{\vep}$, we construct a function $\overline{v} \in \lowers(\overline{B}_{\delta}(x_0))$ with
\begin{equation}
\label{limprop2}
\overline{v} \leq v_{\vep} \quad \text{ in } B_{\delta}(x_0),
\qquad \qquad
\overline{v} = v_{\vep} \quad \text{ on } \partial B_{\delta}^{+}(x_0) \cup \partial B_{\delta}^{-}(x_0),
\end{equation}
satisfying
\begin{equation}
\label{bvp2}
\begin{cases}
F^{\pm}(D^2 \overline{v}) \leq \widetilde{f}^{\pm}(x) +\widetilde{c}_{\vep} & \text{ in } B_{\delta}^{\pm}(x_0),\\
G(\nabla \overline{v}^{+}, \theta \nabla v^{-}) \leq \widetilde{g}^{\pm}(x) +\widetilde{c}_{\vep} & \text{ on } T_{\delta}(x_0),\\
\end{cases}
\end{equation}
in the viscosity sense, and a piecewise linear function $L_{\overline{v}}$ such that
\begin{equation}
\label{lprop2}
|\overline{v} - L_{\overline{v}}| \leq C_{\vep} r^{1+\alpha} \quad \text{ in } B_r(y_0),
\end{equation}
\begin{equation}
\label{eq:key22}
G(\nabla L_{\overline{v}}^{+}, \theta \nabla L_{\overline{v}}^{-}) \leq \widetilde{g}(y_0) + \widetilde{c}_{\vep}.
\end{equation}

\medskip

Let $\overline{w} = \overline{u}- \overline{v}$.
Since $\overline{u}$ and $\overline{v}$ satisfy~\eqref{bvp1} and~\eqref{bvp2}, respectively, by classical comparison principle we have that $\mcal^{+}(D^2 \overline{w};\frac{\lambda}{n}, \Lambda) \geq f^{\pm} - c_{\vep} - (\widetilde{f}^{\pm} + \widetilde{c}_{\vep})$ in $B^{\pm}_{\delta}(x_0)$, in the viscosity sense.
Therefore, noticing that $\overline{w} = w$ on $\partial B^{\pm}_{\delta}(x_0)$ by~\eqref{limprop} and~\eqref{limprop2}, from~\eqref{simple:as} and~\eqref{imp:step} the ABP maximum principle yields
\[
\overline{P} \geq \overline{w} \qquad\text{ in } B_{\delta}(x_0).
\]
It follows that $\overline{P}$ touches $\overline{w}$ from above at $y_{0}$, and from the piecewise linear Taylor expansions~\eqref{lprop} and~\eqref{lprop2} we deduce the following relations for the derivatives: for $i < n$,
\begin{equation}
\label{eq:key}
\overline{P}_{x_i}(y_{0}) =  (L_{\overline{u}})_{x_i} - (L_{\overline{v}})_{x_i}, \qquad \overline{P}_{x_n}^{+}(y_{0}) \geq (L_{\overline{u}})_{x_n}^{+} -(L_{\overline{v}})_{x_n}^{+},\qquad \overline{P}_{x_n}^{-}(y_{0}) \leq (L_{\overline{u}})_{x_n}^{-} - (L_{\overline{v}})_{x_n}^{-}.
\end{equation}

\medskip

By~\eqref{eq:key2}, \eqref{eq:key22}, \eqref{eq:key}, and uniform ellipticity, we conclude
\[
\begin{split}
g(y_0) - c_{\vep} &\leq G(\nabla L_{\overline{u}}^{+}, \theta \nabla L_{\overline{u}}^{-}) 
\\
& \quad \leq G(\nabla L_{\overline{v}}^{+} + \nabla \overline{P}^{+}(y_{0}), \theta \nabla L_{\overline{v}}^{-} + \theta \nabla \overline{P}^{-}(y_{0}))
\\
& \qquad \leq G(\nabla L_{\overline{v}}^{+}, \theta \nabla L_{\overline{v}}^{-})
+ \tcal^{+}(\nabla \overline{P}^{+}(y_{0}), \theta \nabla \overline{P}^{-}(y_{0}))
\\
& \qquad \quad \leq 
\widetilde{g}(y_0) + \widetilde{c}_{\vep}
+ \tcal^{+}(\nabla \overline{P}^{+}(y_0), \theta \nabla \overline{P}^{-}(y_0); \lambda, \Lambda, \mu)
\end{split}
\]
whence we see that
\begin{equation}
\label{eq:key3}
g(y_0) - c_{\vep}  - (\widetilde{g}(y_0) + \widetilde{c}_{\vep})
\leq \tcal^{+}(\nabla \overline{P}^{+}(y_{0}), \theta \nabla \overline{P}^{-}(y_{0}); \lambda, \Lambda, \mu).
\end{equation}
Using that 
$y_0 \in T_{\delta}(x_0)$, 
by continuity of $g$ and $\widetilde{g}$, the left-hand side can be bounded below by
\begin{equation}
\label{eq:key4}
g(y_0) - c_{\vep}  - (\widetilde{g}(y_0) + \widetilde{c}_{\vep})
\geq g(x_0) - c_{\vep}  - (\widetilde{g}(x_0) + \widetilde{c}_{\vep}) - \omega_{g}(\delta)
- \omega_{\widetilde{g}}(\delta).
\end{equation}
Recalling that $\nabla \overline{P}^{\pm} = \nabla P^{\pm} - (a'_{\vartheta}, 0)$, with $|a_{\vartheta}'| < \vartheta$, by uniform ellipticity the right-hand side of~\eqref{eq:key3} is controlled by
\begin{equation}
\label{eq:key5}
\begin{split}
&
\tcal^{+}(\nabla \overline{P}^{+}(y_{0}), \theta \nabla \overline{P}^{-}(y_{0}))
\\
& \quad \leq 
\tcal^{+}(\nabla P^{+}(y_{0}), \theta \nabla P^{-}(y_{0})) 
+ 2 \mu \vartheta\\
& \qquad \leq 
\tcal^{+}(\nabla P^{+}(x_{0}), \theta \nabla P^{-}(x_{0})) 
+ (\Lambda + \mu) \left( 
\|D^2 P^{+}\| + \|D^2 P^{-}\|
\right)|y_0 - x_0|
+ 2 \mu \vartheta\\
& \qquad \quad
\leq 
\tcal^{+}(\nabla P^{+}(x_{0}), \theta \nabla P^{-}(x_{0})) 
+ C \delta
+ 2 \mu \vartheta.
\end{split}
\end{equation}
From~\eqref{eq:key3}, combining~\eqref{eq:key4} and~\eqref{eq:key5} and letting $\vartheta \downarrow 0$ and $\delta \downarrow 0$, we deduce the claim.
\end{proof}

\medskip

\begin{rem}
\label{rem:mistake}
For linear Neumann transmission conditions, an analogous result to Theorem~\ref{thm:comp} was established first for continuous sub- and supersolutions in~\cite{DFSfb,DFS}, and subsequently extended to the semicontinuous setting in~\cite{SS}, by the same methods.
The proof in~\cite{SS} uses two auxiliary functions $\overline{u}^{\vep}$ and $\overline{v}_{\vep}$ (c.f. $\overline{u}$ and $\overline{v}$ above) defined as solutions to certain Dirichlet problems with merely semicontinuous boundary data.
In particular, it is not clear whether the boundary datum is achieved, or in what sense, a point which is crucial to the whole argument.
In our proof above, we avoid this issue by further regularizing the boundary datum $u_k$, then taking replacements $\overline{u}_k$, and finally passing to the limit as $k \to \infty$.
\end{rem}

\medskip

As an immediate corollary of Theorem~\ref{thm:comp}, we obtain the uniqueness of viscosity solutions to the Dirichlet problem in (say) balls:

\begin{cor}
[Uniqueness of viscosity solutions]
Let $\{F^{\pm}, G\}$ be constant-coefficient operators.
Let $f^{\pm} \in C(B_R^{\pm}\cup T_R)$, $g \in C(T_R)$, and $\phi \in C(\partial B_R)$.
Let $0 \leq \theta \leq 1$.

The Dirichlet problem 
\[
\begin{cases}
F^{\pm}(D^2 u, x) = f^{\pm}(x) & \text{ in } B^{\pm}_R,\\
G(\nabla u^{+}, \theta \nabla u^{-}, x) = g(x) & \text{ on } T_R,\\
u = \phi & \text{ on } \partial B_R,
\end{cases}
\]
has at most one viscosity solution $u \in C(\overline{B_R})$.
\end{cor}
\begin{proof}
If $u$ and $v$ are viscosity solutions, then applying Theorem~\ref{thm:comp} to their differences in cubes centered on the interface, we see that $u- v \in S(\lambda,\Lambda,\mu,\theta;0,0,0)$ in $B_R$.
Since $u = v$ on $\partial B_R$, the ABP maximum principle now yields $u = v$ in $B_R$. 
\end{proof}

\medskip

As another corollary, we see that horizontal difference quotients belong to the solution class, a key result for the regularity theory in our work~\cite{E}:
\begin{cor}
Let $u \in C(Q_R)$ be a viscosity solution to
\begin{equation}
\label{eq:hom}
\begin{cases}
F^{\pm}(D^2 u) = 0 & \text{ in } Q^{\pm}_R,\\
G(\nabla u^{+}, \theta \nabla u^{-}) = 0 & \text{ on } T_R.
\end{cases}
\end{equation}
Then, for all $e =(e', 0) \in \R^{n-1} \times \R$, we have that
\[
u(x + e) - u(x) \in S(\lambda,\Lambda,\mu,\theta; 0,0, 0) \quad \text{ in } Q_{R-|e|}.
\]
\end{cor}
\begin{proof}
Since~\eqref{eq:hom} is invariant under horizontal translations, it follows that $u(x)$ and its translation $u(x + e)$ are viscosity solutions in the smaller cube $Q_{R-|e|}$.
Applying Theorem~\ref{thm:comp} to their differences now yields the claim.
\end{proof}

\begin{rem}
We can extend some of the above results to convex equations.
Assuming convexity of the operators $F^{\pm}$ and $G$, arguing similarly to the proof of Theorem~\ref{thm:comp}, we can show that the set of viscosity supersolutions to~\eqref{eq:const} is closed under convex combinations.
As a corollary, we deduce that horizontal second order incremental quotients are in the supersolution class.
More precisely,  if $u$ is a viscosity solution to~\eqref{eq:hom}, then for all $e = (e',0) \in \R^{n-1} \times \R$ we have that
\[
u(x+e) + u(x-e) - 2 u(x) \in \overline{S}(\lambda, \Lambda, \mu, \theta; 0,0,0) \quad \text{ in } Q_{R-|e|}.
\]
This result suggests that horizontal second derivatives are bounded below.
However, we do not know whether a two-sided bound exists, or if the full Hessian is bounded up to the interface.
If true, such a result would have important implications for the regularity theory of (non-convex) variable-coefficient problems; see Section 7 in~\cite{E} for details.
\end{rem}


\medskip

\section{Existence of viscosity solutions via Perron's method}
\label{sec:perron}

Let $F^{\pm}$ and $G$ be constant-coefficient operators with ellipticity constants $0 < \lambda \leq \Lambda$ and $\mu \geq 0$.
Let $f^{\pm} \in C(B_1^{\pm} \cup T_1) \cap L^{\infty}(B_1^{\pm})$ and $g \in C(T_1) \cap L^{\infty}(T_1)$.
Given $\phi \in C(\partial B_1)$, we apply Perron's method to construct a viscosity solution $u \in C(\overline{B}_1)$ to the inhomogeneous transmission problem
\begin{equation}
\label{eq:lul}
\begin{cases}
F^{\pm}(D^2 u) = f^{\pm}(x)& \text{ in } B_1^{\pm},\\
G(\nabla u^{+}, \theta \nabla u^{-}) = g(x) & \text{ on } T_1,
\end{cases}
\end{equation}
satisfying the Dirichlet boundary condition
\begin{equation}
\label{eq:bdy}
u = \phi \quad \text{ on } \partial B_1.
\end{equation}
Our solution is defined as the largest viscosity subsolution, namely, 
we will prove the following:

\begin{thm}[Existence]
\label{thm:existence}
Let $u$ be the function given by 
\[
u(x) =
\sup
\left\{
v(x)\colon
v \in \subsol\right\}
\qquad \text{ for } x \in \overline{B}_1,
\]
where $\subsol$ is the set
\[
\subsol = \left\{
v \in \uppers(\overline{B}_1)
\text{ viscosity subsolution to~\eqref{eq:lul}}, \text{ with } v \leq \phi \text{ on } \partial B_1
\right\}.
\]
Then $u \in C(\overline{B}_1)$ is a viscosity solution to~\eqref{eq:lul} and satisfies the boundary condition~\eqref{eq:bdy}.
\end{thm}

\medskip

Theorem~\ref{thm:existence} will follow from the discussion and three Lemmata~\ref{lem:barrier}, \ref{lem:sub}, and \ref{lem:sup} below.

\medskip

Since the function $u$ need not be continuous (nor semicontinuous) a priori, we must take its \lsc and \usc envelopes $u_{\star}$ and $u^{\star}$, respectively, given by
\[
u_{\star}(x) = \lim_{\rho \downarrow 0} \inf_{B_{\rho}(x) \cap \overline{B}_1} u 
 \qquad \text{ and } \qquad
u^{\star}(x) = \lim_{\rho \downarrow 0} \sup_{B_{\rho}(x) \cap \overline{B}_1} u
\qquad
\text{ for } x \in \overline{B}_1.
\]
Note in particular that
\[
u_{\star} \leq u \leq u^{\star} \quad \text{ on } \overline{B}_1.
\]

\medskip

To show that the boundary condition~\eqref{eq:bdy} is attained, we first construct barriers vanishing at given points on the sphere $\partial B_1$.
As explained in the Introduction, the most delicate case corresponds to when such a point meets the interface.
Here, we use a piecewise smooth barrier inspired by Lieberman's~\cite{L} barrier for the oblique derivative problem based on earlier work of Miller~\cite{Miller}.
For points away from the boundary, we modify a simple barrier of Li-Zhang~\cite{LZ} for the oblique derivative problem.

\begin{lem}
\label{lem:barrier}
For each $x_0 \in \partial B_1$, there is a viscosity supersolution $h \in C(\overline{B}_1)$ to~\eqref{eq:lul} satisfying
\[
h(x_0) = 0
 \qquad \text{ and } 
 \qquad h(x) > 0 \quad \text{ for } x \in \overline{B}_1 \setminus \{x_0\}.
\]
\end{lem}
\begin{proof}
Let $x_0 = (x_0', (x_0)_n) \in \partial B_1 \subset \R^{n-1} \times \R$.
We distinguish two cases according to the relative position of $x_0$ with respect to the interface:

\medskip

\noindent
\textbf{Case 1.} $|(x_0)_n| = \kappa > 0$.

\medskip

Let $v$ and $w$ be the quadratic and piecewise quadratic functions given by
\[
v(x) = 1- (x_0 \cdot x)^2
\qquad \text{ and } 
\qquad w(x) = 3 - (x_n)_{+} - x_n^2,
\]
respectively.
Note that 
\[
v(x_0) = 0,  
\qquad\qquad v > 0 \, \text{ on } \overline{B}_1 \setminus \{x_0\},  
\qquad\qquad v \geq \kappa^2 \, \text{ on } T_1,
\]
\[
w \geq 1 \, \text{ on } \overline{B}_1, \qquad\qquad w \geq 3 \, \text{ on } T_1,
\]
as well as
\[
\begin{cases}
\textstyle
\mcal^{+}(D^2 v; \frac{\lambda}{n}, \Lambda) 
= \mcal^{+}(D^2 w; \frac{\lambda}{n}, \Lambda)
= - 2 \lambda n^{-1} & \text{ in } B_1^{\pm},\\
\tcal^{+}(\nabla w^{+}, \theta \nabla w^{-}; \lambda, \Lambda, \mu)
= - \lambda & \text{ on } T_1.
\end{cases}
\]
Letting
\[
h := \vep^{-1}\min\big\{4 \kappa^{-2} v, w\big\}
\]
and choosing $\vep >0$ sufficiently small, we see from the above that
\begin{equation}
\label{an1}
\textstyle
\mcal^{+}(D^2 h; \frac{\lambda}{n}, \Lambda) 
< - F^{\pm}(0) - \|f^{\pm}\|_{L^{\infty}(B_1^{\pm})} \qquad \text{ in } B_1^{\pm}
\end{equation}
in the viscosity sense, and
\begin{equation}
\label{an2}
\tcal^{+}(\nabla h^{+}, \theta \nabla h^{-}; \lambda, \Lambda, \mu) 
< - G(0,0) - \|g\|_{L^{\infty}(T_1)} \qquad \text{ on } T_1.
\end{equation}
It follows that 
$h$ is a viscosity supersolution to~\eqref{eq:lul} and satisfies the claim.

\medskip

\noindent
\textbf{Case 2.} $(x_0)_n = 0$.

\medskip

For simplicity, throughout the proof, we write $T$ to denote
the ``global'' interface
\[
T = \{x_n = 0\}.
\]
By translation and rotation, we may assume that $B_1 = B_1(e_1)$, $T_1 = T_1(e_1) = T \cap B_1(e_1)$, and $x_0 = 0 \in \partial B_1(e_1)$.
Note here that $e_1 = (1, 0, \ldots, 0)$.

\medskip

We introduce polar coordinates in the $(x_1, x_n)$-plane:
\[
x_1 = r \cos \phi, \qquad x_n = r \sin \phi, \qquad \text{ with } \qquad - \pi/2 \leq \phi \leq \pi/2.
\]
Note that the tangent plane to the convex set $B_1(e_1)$ at $x_0 = 0 \in \partial B_1(e_1)$ is 
$\{x_1 = 0\} \subset \R^{n}$,
hence we always have $x_1 \neq 0$ and therefore
\[
r > 0 \qquad \text{ in } B_1(e_1).
\]
We also note that
\[
T = \{x_n = 0\} = \{\phi = 0\}.
\]

\medskip

Let $w = w(x_1, x_n)$ be the function of two variables given in polar coordinates by
\[
w = r^{\alpha} \zeta(\phi),
\]
where $\zeta$ is the piecewise smooth function
\[
\zeta(\phi)=  2 - e^{\beta (|\phi| - \frac{\pi}{2})}
\]
and $\alpha$, $\beta$ are universal constants given by
\[
\beta = 2 \frac{\Lambda}{\lambda} 
\qquad \text{ and } \qquad 
\alpha = \frac{1}{2} \cdot \frac{\beta}{2 e^{\beta \frac{\pi}{2}}-1} \min\left\{\frac{\lambda}{2\mu}, \beta \right\}.
\]
We point out that $0 < \alpha < 1$ as well as $1 \leq \zeta(\phi) \leq 2$ for $- \frac{\pi}{2} \leq \phi\leq \frac{\pi}{2}$, whence it follows that
\[
r^{\alpha} \leq w \leq 2 r^{\alpha} \quad \text{ on } \overline{B}_1(e_1).
\]
Observe also that $w \in C(\overline{B}_1(e_1))$ is piecewise $C^2$, i.e., its restrictions are $C^2$ up to the interface
\[
w^{\pm} \in C^2(B_1^{\pm}(e_1) \cup T_1(e_1)),
\]
and a fortiori an admissible function.

\medskip

We claim that $w$ is a strict supersolution, in the sense that
\begin{equation}
\label{yor2}
\textstyle
\mcal^{+}(D^2 w; \frac{\lambda}{n}, \Lambda)
\leq 
- c_0 r^{\alpha-2} \quad \text{ in } B_1(e_1)^{\pm},
\end{equation}
\begin{equation}
\label{yor1}
\tcal^{+}(\nabla w^{+}, \theta \nabla w^{-}; \lambda, \Lambda, \mu) 
\leq - c_0 r^{\alpha-1} \quad \text{ on } T_1(e_1),
\end{equation}
where $c_0 > 0$ is the universal constant
\begin{equation}
\label{cos2}
c_0 = \frac{\lambda \beta }{2 e^{\beta\frac{\pi}{2}}}.
\end{equation}

\medskip

Let us check~\eqref{yor2} first.
By Miller's polar representation~\cite{Miller}, every non-divergence operator
$L w = \tr{A(x)D^2 w}$
acting on $w(x_1, x_n)$, with $A(x) \in \ellmat$, can be written as
\[
L w = r^{\alpha-2} \left[a(x) \alpha(\alpha-1) \zeta + 2 b(x) (\alpha-1) \zeta' + c(x) (\zeta'' + \alpha \zeta) \right],
\]
for some $2\times2$ elliptic matrix 
$\left(\begin{smallmatrix}
a(x) & b(x) \\ c(x) & d(x)
\end{smallmatrix}\right) \in \ellmat$.
Observing that
$0 < \alpha < 1$, $\zeta' < 0$, and $\zeta'' + \alpha \zeta < 0$,
a computation shows that
\[
\begin{split}
\mcal^{+}(D^2 w) &= r^{\alpha-2} \left[ -\lambda \alpha (1-\alpha) \zeta - (\Lambda-\lambda) (1-\alpha) \zeta' + \lambda (\zeta'' + \alpha \zeta)  \right]\\
&= r^{\alpha-2} e^{\beta(|\phi| - \frac{\pi}{2})}  \left[ \lambda \big(- \beta^2 + \alpha^2 (2 e^{\beta(\frac{\pi}{2}-|\phi|)}- 1)\big) + \beta (\Lambda-\lambda) (1-\alpha) \right]\\
& \qquad \leq r^{\alpha-2} \lambda e^{\beta(|\phi| - \frac{\pi}{2})}  \left[- \beta^2 + \alpha^2 (2 e^{\beta\frac{\pi}{2}}- 1) + \beta \frac{\Lambda-\lambda}{\lambda} \right]\\
&\qquad \qquad \leq r^{\alpha-2} \lambda \beta e^{\beta(|\phi| - \frac{\pi}{2})}  \left[ - \frac{1}{2}\beta + \frac{\Lambda}{\lambda}-1 \right]\\
&\qquad \qquad \qquad \leq  - r^{\alpha-2} \lambda \beta e^{- \beta\frac{\pi}{2}} \leq - c_0 r^{\alpha-2},
\end{split}
\]
where in the last line we have used the definition of $\beta$ and $c_0$ in~\eqref{cos2}.

\medskip

Next, we verify~\eqref{yor1}.
A direct computation shows that the gradient of $w$ away from the interface $T = \{\phi = 0\}$ is given by
\[
\nabla w = 
r^{\alpha-1} \Big[
\left( \alpha \zeta(\phi) \cos \phi - \zeta'(\phi) \sin \phi\right) e_1 +
\left( \alpha \zeta(\phi) \sin \phi + \zeta'(\phi) \cos \phi\right) e_n
\Big],
\]
and hence, restricted to the interface $T_1(e_1)$ from each side
\[
\nabla w^{\pm} 
= 
r^{\alpha-1} \Big[
 \alpha \zeta(0) e_1 
+[\zeta^{\pm}]'(0) e_n
\Big] \quad \text{ on } T_1(e_1).
\]
Since $[\zeta^{+}]'(0) < 0$ and $[\zeta^{-}]'(0) > 0$, we see that $w_{x_n}^{+} < 0$ and $w_{x_n}^{-} > 0$ on $T$, and hence by~\eqref{tmin}
\[
\begin{split}
\tcal^{+}(\nabla w^{+}, \theta \nabla w^{-}; \lambda, \Lambda, \mu)
&
= \lambda (w_{x_n}^{+} - \theta w_{x_n}^{-}) + \mu (1+ \theta) |w_{x_1}|\\ &  \qquad \leq \lambda w_{x_n}^{+} + 2 \mu |w_{x_1}|
= r^{\alpha-1} \left(\lambda [\zeta^{+}]'(0)  + 2\mu \alpha \zeta(0)\right)
\end{split}
\]
whence, recalling~\eqref{cos2}, we deduce
\[
\begin{split}
\tcal^{+}(\nabla w^{+}, \theta \nabla w^{-}) & \leq r^{\alpha-1} \left(
- \lambda \beta e^{-\beta\frac{\pi}{2}}  + 2\mu \alpha (2- e^{-\beta\frac{\pi}{2}})
\right) 
\leq - c_0 r^{\alpha-1} \quad \text{ on } 
T_1(e_1),
\end{split}
\]
which was the claim.

\medskip

Finally, we extend $w$ in the remaining variables $x = (x_1, x'', x_n) \in B_1(e_1) \subset \R \times \R^{n-2} \times \R$ by
\begin{equation}
\label{hfunction}
h(x) := w(x_1, x_n) + \frac{1}{2} c_1 |x''|^2, 
\end{equation}
where $c_1 > 0$ is the universal constant
\[
c_1 := 
2^{\alpha-2}
\frac{c_0}{2} \min\left\{\frac{1}{\Lambda n}, \frac{1}{\mu}\right\}.
\]
Since $r \leq 2$ in $B_1(e_1)$, from~\eqref{yor2} we obtain
\[
\begin{split}
\textstyle
\mcal^{+}(D^2 h; \frac{\lambda}{n}, \Lambda) &\leq \mcal^{+}(D^2 w) + c_1 \textstyle\mcal^{+}(D^2 \frac{1}{2} |x''|^2)\\
& \quad \leq - 2^{\alpha-2} c_0 + \Lambda(n-2) c_1 \\
& \qquad \leq - 2^{\alpha-3} c_0
\qquad  \text{ in } B_1(e_1)^{\pm},
\end{split}
\]
and from~\eqref{yor1} we deduce
\[
\begin{split}
\tcal^{+}(\nabla h^{+}, \theta \nabla h^{-}; \lambda, \Lambda, \mu) 
&\leq \tcal^{+}(\nabla w^{+}, \theta \nabla w^{-}) + c_1 \textstyle \tcal^{+}(\nabla \frac{1}{2} |x''|^2, \theta \nabla \frac{1}{2} |x''|^2)\\
& \quad \leq - 2^{\alpha-1}c_0 + 2\mu c_1\\
& \qquad \leq - 2^{\alpha-2} c_0 
\qquad \text{ on } T_1(e_1),
\end{split}
\]

In particular, the last two inequalities show that the function $h$ defined in~\eqref{hfunction} satisfies
\begin{equation}
\label{hfunction2}
\begin{cases}
\mcal^{+}(D^2 h; \frac{\lambda}{n}, \Lambda) \leq - c_2 & \text{ in } B_1(e_1)^{\pm},\\
\tcal^{+}(\nabla h^{+}, \theta \nabla h^{-}; \lambda, \Lambda, \mu) \leq - c_2 & \text{ on } T_1(e_1),\\
\end{cases}
\end{equation}
with $c_2 = 2^{\alpha-2} c_0 > 0$.
Replacing $h$ by $\vep^{-1} h$, from~\eqref{hfunction2} we see that we may take $\vep > 0$ sufficiently small such that \eqref{an1} and~\eqref{an2} hold, whence the claim follows.
\end{proof}

\medskip

For $x_0 \in \partial B_1$ and $\vep >0$, we define the lower barrier
\[
\underline{\varphi}(x) = \phi(x_0) - \vep - C_{\vep} h(x),
\]
where $h$ is given by Lemma~\ref{lem:barrier} and $C_{\vep} \geq 1$ is chosen large such that $\underline{\varphi} \leq \phi$ on $\partial B_1$.
Since $\underline{\varphi} \in \subsol$ is a continuous function, by definition of $u$  and \lsc envelope, we have that
\[
\underline{\varphi} \leq u_{\star} \leq u \leq u^{\star} \quad \text{ on } \overline{B}_1.
\]
In particular, we see that $\phi(x_0) \leq u_{\star}(x_0) + \vep$ for all $x_0 \in \partial B_1$ and $\vep > 0$, whence it follows that
\begin{equation}
\label{eq:one3}
\phi \leq u_{\star} \leq u
\qquad \text{ on } \partial B_1.
\end{equation}

\medskip

Similarly, the upper barrier
\[
\overline{\varphi}(x) = \phi(x_0) + \vep + C_{\vep} h(x)
\]
is a continuous supersolution satisfying $\overline{\varphi} \geq \phi$ on $\partial B_1$.
By comparison principle (Theorem~\ref{thm:comp}), for all $v \in \subsol$, we have that $v - \overline{\varphi} \in \underline{S}(\lambda, \Lambda, \mu, \theta; 0,0,0)$ in $B_1$ and hence the ABP yields $v \leq \overline{\varphi}$ in $B_1$.
Recalling the definition of $u$ and \usc envelope, we deduce $u \leq u^{\star} \leq \overline{\varphi}$ in $B_1$ and arguing as above, we conclude
\begin{equation}
\label{eq:one5}
u \leq u^{\star} \leq \phi \qquad \text{ on } \partial B_1.
\end{equation}
Combining~\eqref{eq:one3} and~\eqref{eq:one5}, we see that the boundary condition~\eqref{eq:bdy} is attained with
\begin{equation}
\label{bdy:thing}
u_{\star} = u = u^{\star} = \phi \qquad \text{ on } \partial B_1.
\end{equation}

\medskip

Next, we show that $u^{\star}$ is a subsolution:

\begin{lem}
\label{lem:sub}
The \usc envelope $u^{\star}$ is a viscosity subsolution to~\eqref{eq:lul}
\end{lem}
\begin{proof}
We only check the transmission condition.
Let $P$ be a piecewise quadratic function touching $u^{\star}$ from above at $x_0 \in T_1$.
By Lemma~\ref{lem:trick}, we may assume that
\begin{equation}
\label{truco}
F^{\pm}(D^2 P^{\pm}) < - \|f^{\pm}\|_{L^{\infty}}(B^{\pm}_{r_0}(x_0)),
\end{equation}
for a small $r_0 >0$.

\medskip

Fix $\delta >0$ small.
By definition of $u^{\star}$ and $u$, a vertical translation of $P_{\delta}(x) := P(x) + \frac{1}{2}\delta |x-x_0|^2$ touches an \usc subsolution $v_k$ from above at $x_k \in B_{r_k}(x_0)$, with $r_k \downarrow 0$.
Taking $\delta >0$ sufficiently small, since $v_k$ are subsolutions, by~\eqref{truco} we have that $x_k \in T_{r_0}(x_0)$ for infinitely many $k$ and at those points
\[
g(x_k) \leq G(\nabla P_{\delta}^{+}(x_k), \theta \nabla P_{\delta}^{-}(x_k))
\leq 
G(\nabla P^{+}(x_k), \theta \nabla P^{-}(x_k))
+ 2 \delta \mu r_k.
\]
Since $x_k \to x_0$ and $r_k \to 0$ as $k \to \infty$, passing to the limit yields the claim.
\end{proof}

\medskip

Combining~\eqref{bdy:thing} and Lemma~\ref{lem:sub}, we see that $u^{\star} \in \subsol$, and by definition of $u$ we conclude
\begin{equation}
\label{u:usc}
u = u^{\star} \quad \text{ on } \overline{B}_1.
\end{equation}

\medskip

To finish the argument, it remains to show that $u_{\star}$ is a supersolution:

\begin{lem}
\label{lem:sup}
The \lsc envelope $u_{\star}$ is a viscosity supersolution to~\eqref{eq:lul}.
\end{lem}
\begin{proof}
Otherwise, there is a piecewise quadratic function $P$ touching $u_{\star}$ from below at $x_0 \in T_1$ (the case $x_0 \notin T_1$ is known) and violating the transmission condition at that point, whence by continuity we obtain
\begin{equation}
\label{eq:que}
G(\nabla P^{+}(x), \theta \nabla P^{-}(x))  - g(x) \geq \eta > 0 \qquad \text{ for all } x \in T_{r}(x_0),
\end{equation}
for some small $r > 0$ and $\eta > 0$.
Arguing as in the proof of Lemma~\ref{lem:trick} we may assume that
\begin{equation}
\label{eq:subb0}
F^{\pm}(D^2 P^{\pm}) >  \|f^{\pm}\|_{L^{\infty}(B^{\pm}_{r}(x_0))}.
\end{equation}

\medskip

For $\delta > 0$, we modify $P$ by $P_{\delta}(x) := P(x) - \frac{1}{2}\delta |x-x_0|^2 + \frac{1}{4}\delta r^2$.
By~\eqref{eq:que},~\eqref{eq:subb0}, and uniform ellipticity, we can choose $\delta> 0$ sufficiently small such that
\[
G(\nabla P_{\delta}^{+}(x), \theta \nabla P_{\delta}^{-}(x)) 
\geq G(\nabla P^{+}(x), \theta \nabla P^{-}(x)) -2 \mu \delta r 
\geq g(x) + \eta -2 \mu \delta r
> g(x)
\]
for all $x \in T_r(x_0)$, and
\[
F^{\pm}(D^2 P^{\pm}_{\delta}) 
\geq F^{\pm}(D^2 P^{\pm}) - \Lambda \delta 
> \|f^{\pm}\|_{L^{\infty}(B^{\pm}_{r}(x_0))}.
\]
Therefore, $P_{\delta}$ is a strict subsolution to~\eqref{eq:lul} in $B_{r}(x_0)$.

\medskip

Since $P_{\delta} < u_{\star} \leq u$ on $\partial B_r(x_0)$, it follows that the function
\[
v = \begin{cases}
\max\{u, P_{\delta}\} & \text{ in } B_r(x_0)\\
u & \text{ on } \overline{B}_1 \setminus B_r(x_0)
\end{cases}
\]
is in the set~$\subsol$.
However, since $P_{\delta}(x_0) > u_{\star}(x_0) = \lim_{\rho\downarrow 0} \inf_{B_\rho(x_0)} u$, we can find a point $y \in B_r(x_0)$ where $u(y) < P_{\delta}(y) = v(y)$, contradicting the definition of $u$.
We conclude that no such piecewise quadratic function $P$ can exist, hence $u_{\star}$ is a subsolution.
\end{proof}

\medskip

Thanks to~\eqref{u:usc} and Lemma~\ref{lem:sup}, by comparison principle $u - u_{\star} \in \underline{S}(\lambda, \Lambda, \mu, \theta; 0,0,0)$ in $B_1$, whence by~\eqref{bdy:thing} and the ABP we deduce that $u \leq u_{\star}$ on $\overline{B}_1$.
Therefore, $u = u^{\star} = u_{\star} \in C(\overline{B}_{1})$ is a viscosity solution to~\eqref{eq:lul}.


\section*{Acknowledgments}

The author would like to thank Mar\'{i}a Soria-Carro and Kai Zhang for helpful discussions.


\begin{thebibliography}{10}

\bibitem{B}
M. Borsuk.
\textit{Interface problems for elliptic second-order equations in non-smooth domains}.
Second edition. Front. Math. Birkh\"{a}user/Springer, Cham, 2024.

\bibitem{CEF}
X.~Cabr\'{e}, I.~U.~Erneta, and J.C.~Felipe-Navarro.
Null-Lagrangians and calibrations for general nonlocal functionals and an application to the viscosity theory.
\textit{J. Funct. Anal.} 289 (2025), no. 9, Paper No. 111086.

\bibitem{CC}
L.~A.~Caffarelli and X.~Cabr\'{e}.
\emph{Fully nonlinear elliptic equations}.
Amer. Math. Soc. Colloq. Publ. 43, American Mathematical Society, Providence, RI, 1995.

\bibitem{CIL}
M.~G.~Crandall, H.~Ishii, and P.-L.~Lions.
User's guide to viscosity solutions of second order partial differential equations.
\textit{Bull. Amer. Math. Soc.} (N.S.) 27 (1992), no.~1, 1--67.

\bibitem{DFSfb}
D.~De Silva, F.~Ferrari, and S.~Salsa.
Free boundary regularity for fully nonlinear non-homogeneous two-phase problems.
\textit{J. Math. Pures Appl.} (9) 103 (2015), no.~3, 658--694.

\bibitem{DFS}
D.~De Silva, F.~Ferrari, and S.~Salsa.
Regularity of transmission problems for uniformly elliptic fully nonlinear equations.
Proceedings of the International Conference ``Two nonlinear days in Urbino 2017".
\textit{Electron. J. Differ. Equ. Conf.} 25 (2018), 55--63.

\bibitem{E}
I.~U.~Erneta.
Regularity theory for fully nonlinear oblique transmission problems.
Preprint arXiv 2606.14858 (2026).

\bibitem{MW}
F.~Ma and L.~Wang.
Boundary first order derivative estimates for fully nonlinear elliptic equations.
\textit{J. Differential Equations} 252 (2012), no.~2, 988--1002.

\bibitem{L} 
G.~M.~Lieberman.
Mixed boundary value problems for elliptic and parabolic differential equations of second order.
\textit{J. Math. Anal. Appl.} 113 (1986), no.~2, 422--440.

\bibitem{LPerron}
G.~M.~Lieberman.
The Perron process applied to oblique derivative problems.
\textit{Adv. in Math.} 55 (1985), no.~2, 161--172.

\bibitem{LZ}
D.~Li and K.~Zhang.
Regularity for fully nonlinear elliptic equations with oblique boundary conditions.
\textit{Arch. Ration. Mech. Anal.} 228 (2018), no.~3, 923--967.

\bibitem{JLL}
J.-L.~Lions.
Contribution \`{a} un probl\`{e}me de M. M. Picone.
\textit{Ann. Mat. Pura Appl.} 41 (1956), 201--219.

\bibitem{Miller}
K. Miller.
Barriers on cones for uniformly elliptic operators.
\textit{Ann. Mat. Pura Appl.} 76 (1967), 93--105.

\bibitem{P} 
M.~Picone.
Sur un probl\`{e}me nouveau pour l'\'{e}quation lin\'{e}aire aux d\'{e}riv\'{e}es partielles de la th\'{e}orie math\'{e}matique classique de l'\'{e}lasticit\'{e}.
Second \textit{Colloque sur les \'{e}quations aux d\'{e}riv\'{e}es partielles}, Bruxelles (1954).

\bibitem{SS}
M.~Soria-Carro and P.~R.~Stinga.
Regularity of viscosity solutions to fully nonlinear elliptic transmission problems.
\textit{Adv. Math.} 435 (2023), Paper No. 109353.

\end{thebibliography}
\end{document}